\documentclass[11pt]{article}
\usepackage{graphicx}
\usepackage{amsthm, amsmath, amssymb, tikz, bm}
\usepackage{mathtools, intcalc}
\usepackage{xifthen}
\usepackage{comment}
\usepackage[colorlinks=true,
linkcolor=blue,citecolor=blue,
urlcolor=blue]{hyperref}
\usepackage{blkarray}

\usepackage[margin=1.2in]{geometry} 

\usepackage[shortlabels]{enumitem}
\usepackage{todonotes}
\usetikzlibrary{calc, shapes, backgrounds}
\allowdisplaybreaks

\newtheorem{lemma}{Lemma}
\newtheorem{theorem}{Theorem}
\newtheorem{corollary}{Corollary}
\newtheorem{definition}{Definition}

\newtheorem{conjecture}{Conjecture}
\newtheorem{proposition}{Proposition}

\newcommand{\ds}{\displaystyle}
\newcommand{\dss}{\displaystyle\sum}

\newcommand{\E}{{\rm E}}
\newcommand{\lp}{\left(}
\newcommand{\rp}{\right)}

\newcommand{\lcm}{{\rm lcm}}

\usepackage[normalem]{ulem}

\title{Ricci Curvature Formula: Applications to Bonnet-Myers Sharp Irregular Graphs}
\author{Yupei Li
\thanks{University of South Carolina, Columbia, SC 29208, ({\tt yupei@email.sc.edu})}
 \and Linyuan Lu \thanks{University of South Carolina, Columbia, SC 29208,
({\tt lu@math.sc.edu}).
This author was supported in part by NSF grant DMS 2038080.
}  
}

\begin{document}
\maketitle

\begin{abstract}
In this paper, we establish a simple formula for computing the Lin-Lu-Yau Ricci curvature on graphs. For any edge $xy$ in a simple locally finite graph $G$, the curvature $\kappa(x,y)$ can be expressed as a cost function of an optimal bijection between two blow-up sets of the neighbors of $x$ and $y$. Utilizing this approach, we derive several results including a structural theorem for the Bonnet-Myers sharp irregular graphs of diameter $3$ and a theorem on $C_3$-free
Bonnet-Myers sharp graphs.
\end{abstract}
\section{Introduction}
In Riemannian geometry, the renowned Bonnet–Myers theorem \cite{Myers1941} asserts that if the Ricci curvature of a complete Riemannian manifold $M$ is bounded below by $(n-1)\kappa > 0$, then its diameter is at most $\frac{\pi}{\sqrt{\kappa}}$. Over the past three decades, there have been numerous efforts to extend the concept of Ricci curvature from Riemannian geometry to 
a discrete setting.
There  are different definitions of  Ricci curvature defined on graphs,  see references \cite{Jost, Fchung, LYau, Ollivier1}.
For instance, Ollivier \cite{Ollivier} introduced a notion of Ricci curvature in metric spaces equipped with a measure of a random walk. Lin-Lu-Yau \cite{LLY} defined a variation of Ollivier-type Ricci curvature on graphs. These notions allow for the generalization of several classical theorems associated with positive Ricci curvature, such as the Lichnerowicz theorem, the Bonnet–Myers theorem, and many others. 
Numerous properties and implications of the Ollivier Ricci curvature and Lin-Lu-Yau curvature have been explored,  see \cite{BCLMP, BJL, BBS,  JL, Smith}, etc. These curvatures has been applied in various research areas such as network
analysis\cite{NLLG, JJBP}, quantum computation, dynamic Networks\cite{p-adic}, etc.

In this paper, we consider locally finite graphs, where each vertex has a finite degree, but the total number of vertices may be countably infinite.
Let $G=(V,E)$ be a simple locally finite graph with the vertex set $V$ and edge set $E$. (Here  “simple” refers to the absence of self-loops and multi-edges.) For any vertex $v\in V(G)$, let $N_G(v)$ denote the neighborhood of $v$ in $G$, i.e., $N_G(v) =\{u : vu\in E(G)\}$, and let $N_G[v]:= N_G(v)\cup \{v\}$ denote the closed neighborhood of $v$.  The degree of a vertex $v$, denoted by $d_v$, is the number of the neighbors of $v$, i.e.,  $d_v=|N_G(v)|$. 
A graph $G$ is called {\em $d$-regular} if $d_v=d$ for any vertex $v$. We call $G$ {\em irregular} if it is not regular.
For  $S\subseteq V(G)$, $N_G(S) = \{u \in V(G)\setminus S: us\in E(G) \textrm{ for some } s \in S\}$. 
When $G$ is clear under context, we will omit the subscript $G$ in the notation.

A $uv$-walk $P$ is a sequence of vertices $u=u_0, u_1, \ldots, u_k=v$ such that $u_i$ and $u_{i+1}$ are adjacent for $i=0,\ldots, k-1$. Sometime we write $uPv$ to indicate the walk is from $u$ to $v$. Given two walks $uP_1v$ and $vP_2w$, the concatenation of two walk is a new walk, written by $uP_1vP_2w$, that walks from $u$ to $v$ via $P_1$ and then continues from $v$ to $w$ via $P_2$. A walk $P$ is called $uv$-path if all internal vertices are distinct. A walk is called {\em closed} if $u=v$. A {\em cycle} is a closed walk so that all internal vertices are distinct. We say $G$ is {\em connected} if for any pair of vertices $u$ and $v$, there is a path (or walk) $uPv$.

For any two vertices $u,v \in V(G)$, the \textit{distance} from $u$ to $v$ in $G$, denoted by $d_G(u,v)$, is the number of edges of a shortest path from $u$ to $v$ in $G$. 
The \textit{diameter} of a graph $G$ is defined to be $diam(G)=\sup\{d_G(u,v): u,v\in V(G)\}$. 

A mass distribution $m$ (over the vertex set $V=V(G)$) is a mapping $m: V\to [0,\infty)$. The total mass of $m$ is the $L_1$ norm, 
$\|m\|_1=\sum_{v\in V}m(v)$. A probability distribution is a mass distribution with total mass equal to 1.

Let $m_1$ and $m_2$ be two mass distributions on $V$ with equal total masses.
A {\em coupling} between $m_1$ and $m_2$ is a mapping $A: V\times V \to [0,1]$  such that 
\begin{equation}\label{eq:coupling}
\dss_{y \in V} A(x,y) = m_1(x) \textrm{ and } \dss_{x\in V} A(x,y) = m_2(y).
\end{equation}
The cost of $A$, denoted by $C(A)$, is given by
\begin{equation}
    C(A)=\sum_{u,v}A(u,v)d(u,v).
\end{equation}
The \textit{transportation distance} between the two probability distributions $m_1$ and $m_2$ is defined as follows:
\begin{equation}\label{eq:distance-coupling}
W(m_1, m_2) = \inf_A C(A),
\end{equation}
where the infimum is taken over all coupling $A$ between $m_1$ and $m_2$. By the duality theorem of a linear optimization problem, the transportation distance can also be expressed as follows:
\begin{equation}\label{eq:distance-Lipschitz}
W(m_1, m_2) = \sup_f \dss_{x\in V} f(x) \lp m_1(x)-m_2(x)\rp,
\end{equation}
where the supremum is taken over all $1$-Lipschitz functions $f$.

A \textit{random walk} $m$ on $G=(V,E)$ is defined as a family of probability measures $\{m_v(\cdot)\}_{v\in V}$ such that $m_v(u) = \frac{1}{d_v}$ for any $u\in N(v)$ and $0$ otherwise.
The \textit{Ricci curvature} $\kappa: \binom{V(G)}{2} \to \mathbb{R}$ of $G$ can then be defined as follows:

\begin{definition}
Given a connected locally finite graph $G=(V,E)$, a random walk $m = \{m_v(\cdot)\}_{v\in V}$ on $G$ and two vertices $x,y\in V$, define the Ollivier Ricci Curvature
$$\kappa_0(x,y) = 1 - \frac{W(m_x, m_y)}{d(x,y)}.$$
\end{definition}

For $0\leq \alpha < 1$, the {\em $\alpha$-lazy random walk} $m_x^{\alpha}$ (for any vertex $x$), is defined as 
\[
m_x^{\alpha}(v) = \begin{cases} 
                        \alpha & \textrm{ if $v=x$,}\\
                        (1-\alpha)/d_x &\textrm{ if $v\in N(x)$,}\\
                        0 & \textrm{ otherwise.}
                    \end{cases}
\]

In \cite{LLY}, Lin, Lu, and Yau defined the Ricci curvature of graphs based on the $\alpha$-lazy random walk as $\alpha$ goes to $1$. More precisely,
for any $x,y \in V$, they defined the $\alpha$-Ricci-curvature $\kappa_{\alpha}(x,y)$ to be 
$$\kappa_{\alpha}(x,y) = 1 - \frac{W(m_x^{\alpha}, m_y^{\alpha})}{d(x,y)}$$ and the Lin-Lu-Yau Ricci curvature $\kappa_{\textrm{LLY}}$ of $G$ to be 
\[\kappa_{\textrm{LLY}}(x,y) = \ds\lim_{\alpha \to 1} \frac{\kappa_{\alpha}(x,y)}{(1-\alpha)}.\]
This limit always exists since $\kappa_{\alpha}$ is concave in $\alpha \in [0,1]$ for any two vertices $x,y$ (see \cite{LLY}). 

The parameter $\alpha$ is called the {\em idleness} of the random walk.  
When $\alpha=0$, $m_x^{\alpha}$ becomes classical random walk. The curvature $\kappa_0$ is  Ollivier's original definition of Ricci curvature on graphs (see \cite{Ollivier}). It has been shown in \cite{LLY} that
\begin{equation} \label{eq:relation}
\kappa_{LLY}(x,y)\geq \kappa_0(x,y)
\end{equation}
for any pair of vertices $x,y$.

In this paper, we only consider the Lin-Lu-Yau curvature and simply write
$\kappa_{LLY}(x,y)$ as $\kappa(x,y)$ for the rest of the paper.
Although the Ricci curvature $\kappa(x,y)$ is defined for all pairs $x,y \in V(G)$, it suffices to consider only $\kappa(x,y)$ for $xy\in E(G)$ due to the following lemma.

\begin{lemma}\cite{LLY, Ollivier}\label{lem:adj-pair}
Let $G$ be a connected graph. If $\kappa(x,y) \geq k$ for any edge $xy\in E(G)$, then $\kappa(x,y) \geq k$ for any pair of vertices $(x,y)$.
\end{lemma}

Bourne-Cushing-Liu-M\"unch-Pyerimhoff proved \cite{BCLMP2018} the following result:

\begin{theorem}(see \cite{BCLMP2018}, Theorem 1.1)
Let $G = (V, E)$ be a locally finite graph. For any edge $xy$,
the function $\alpha \to \kappa_\alpha(x, y)$ is concave and piece-wise linear over $[0, 1]$ with at most 3 linear parts. Furthermore $\kappa_\alpha(x, y)$ is linear on the intervals
$[0, \frac{1}{\lcm(d_x, d_y) + 1}]$
and $[\frac{1}{\max(d_x, d_y) + 1}, 1)$
Thus, if we have the further condition $d_x = d_y$, then $\kappa_\alpha(x, y)$ has at most two linear parts.
\end{theorem}
Their result implies
\begin{equation}\label{eq:BCL}
  \kappa(x,y)=\frac{1}{(1-\alpha)}\kappa_\alpha, \hspace*{1cm} \mbox{ for }\alpha\in \left[\frac{1}{\max\{d_x,d_y\}+1},1\right).
\end{equation}

M\"{u}nch and Wojciechowski \cite{MW} gave another limit-free formulation of the Lin-Lu-Yau Ricci curvature using \textit{graph Laplacian}. For a graph $G = (V,E)$, the (negative) combinatorial graph Laplacian $\Delta$ is defined as: 
$$\Delta f(x)=\frac{1}{d_x}\sum\limits_{y\in N(x)} (f(y)-f(x)). $$

\begin{theorem}\label{thm:curvature_laplacian}\cite{MW} (Curvature via the Laplacian) Let $G$ be a simple graph and let $x \neq y \in V(G)$. Then 
\begin{equation} \label{eq:curv_laplacian}
    \kappa(x,y) = \inf_{\substack{f \in Lip(1)\\ \nabla_{yx}f = 1}} \nabla_{xy} \Delta f, 
\end{equation}
where $\nabla_{xy}f=\frac{f(x)-f(y)}{d(x,y)}$. 
\end{theorem}
It was also proved (see \cite{MW, BCLMP, CK})
that the optimal solution $f$ in \eqref{eq:curv_laplacian} can be chosen to be an integer-valued function.

Bai, Huang, Lu, and Yau prove a dual theorem for a limit-free definition for the Lin-Lu-Yau Ricci curvature. 
For any two vertices $x$ and $y$, a {\em $\ast$-coupling} between $m^0_x$ and $m^0_y$ is a mapping $B: V\times V\to \mathbb{R}$ with finite support such that 
\begin{enumerate}
    \item $0<B(x,y)$, but all other values $B(u,v)\leq 0$.
    \item $\sum\limits_{u,v\in V} B(u,v)=0$.   
    \item $\sum\limits_{v \in V} B(u, v)=-m^0_x(u)$ for all $u$ except $x$.
    \item $\sum\limits_{u \in V} B(u, v)=-m^0_y(v)$ 
for all $v$ except $y$.  
\end{enumerate}

\begin{theorem}\cite{BHLY} (Curvature via the $\ast$-Coupling function)\label{thm:curvatureviacoupling}
For any two vertex $x, y\in V(G)$, we have
\begin{equation} \label{eq:curv_coupling}
\kappa(x,y)=\frac{1}{d(x,y)}\sup\limits_{B} \sum\limits_{u, v\in V} B(u, v)d(u, v),
\end{equation}
where the superemum is taken over all $\ast$-coupling $B$ between $m^0_x$ and $m^0_y$.
\end{theorem}

In this paper, we will derive a straightforward formula to calculate $\kappa(x,y)$ for any
edge $xy\in E(G)$. We use the following notation. Let $\lcm(d_x,d_y)$ denote the least common multiplier of $d_x$ and $d_y$. 
Let $c_x$ and $c_y$ are a pair of relative prime integers such that
\[ \lcm(d_x,d_y)=c_xd_x=c_yd_y.\]

Without loss of generality, we assume $d_x\leq d_y$. Equivalently, we have $c_x\geq c_y$. 
Let $\mu_x$ be a mass distribution defined as
\[
\mu_x(u)=\begin{cases}
c_x-c_y & \mbox{ if } u\in N(x)\cap N[y],\\
   c_x &\mbox{ if } u \in N(x)\backslash N[y],\\
0 &\mbox{otherwise.}   
\end{cases}
\]
Let $\mu_y$ be a mass distribution defined as
\[
\mu_y(u)=\begin{cases}
   c_y &\mbox{ if } u \in N(y)\backslash N[x],\\
0 &\mbox{otherwise.}   
\end{cases}
\]

For any coupling $\sigma$ between $\mu_x$ and $\mu_y$, recall the cost function $C(\sigma)$
is
\begin{equation}
    C(\sigma)=\sum_{u,v} \sigma(u,v)d(u,v).
\end{equation}
In the expression of the cost function 
$C(\sigma)$ above, it suffices to sum up only for $u\in N(x)$ and $v\in N(y)\setminus N[x]$ because $\sigma(u,v)=0$ if $(u,v)\not\in N(x)\times (N(y)\setminus N[x])$.

We have the following theorem.
\begin{theorem}\label{thm:localstructure}
For any edge $xy\in E(G)$, assuming $d_x\leq d_y$, we have
\begin{equation}\label{eq:formula}
    \kappa(x,y)= 1+\frac{1}{d_y}-\frac{\min_{\sigma} C(\sigma)}{\lcm(d_x, d_y)}.
\end{equation}
Here the minimum is taken over all integer-valued couplings $\sigma$ between $\mu_x$ and $\mu_y$.
\end{theorem}

For example, consider the following graph in Figure \ref{fig:example1}. We have $d_x=3$ and $d_y=4$. Thus, $\lcm(d_x,d_y)=12$, $c_x=4$, and $c_y=3$. 
The mass distribution $\mu_x$ is labelled in red color while $\mu_y$ is labelled in black color.  
Any coupling between $\mu_x$ and $\mu_y$ has a cost of at least $14$ since  transferring $4$ units from $x_1$ to $y_1$ or $y_2$ costs at least $4\times 3=12$ and moving other two units has the cost of at least $2$. Clearly, an optimal coupling with cost $14$ exists. For example, one can transfer
one unit from $z$ to $y_1$, one unit from $y$ to $y_2$, two units from $x_1$ to $y_1$, and two units from $x_1$ to $y_2$. 
So we have $$\min_\sigma C(\sigma)=14.$$
By Theorem \ref{thm:localstructure}, we get
\[\kappa(x,y) = 1 + \frac{1}{4}-\frac{14}{12}=\frac{1}{12}.\]
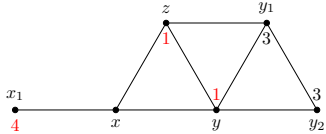
\begin{figure}[htb]
\begin{center}
     \resizebox{4.5cm}{!}{\begin{tikzpicture}[scale=1, Wvertex/.style={circle, draw=black, fill=white, scale=3}, bvertex/.style={circle, draw=black, fill=black, scale=0.3},rvertex/.style={circle, draw=red, fill=red, scale=0.2}]

\node [bvertex, label={[font=\small] below:$x$}] (x) at (-1,0) {};
\node [bvertex, label={[font=\small] below:$y$}, label={[font=\small] above:\textcolor{red}{$1$}}]  (y) at (1,0) {};
\node [bvertex, label={[font=\small] above:$z$}, label={[font=\small] below:\textcolor{red}{$1$}}] (z) at (0,1.73) {};
\node [bvertex, label={[font=\small] above:$y_1$}, label={[font=\small] below:$3$}] (y1) at (2,1.73) {};
\node [bvertex, label={[font=\small] below:$y_2$}, label={[font=\small] above:$3$}] (y2) at (3,0) {};
\node [bvertex, label={[font=\small] above:$x_1$}, label={[font=\small] below:\textcolor{red}{$4$}}] (x1) at (-3, 0) {};

\draw (x) -- (y);
\draw (x) -- (z);
\draw (y) -- (z);
\draw (z) -- (y1);
\draw (y) -- (y1);
\draw (y) -- (y2);
\draw (x) -- (x1);
\draw (y1) -- (y2);

\end{tikzpicture}	}    
\end{center}
    \caption{A example: Mass distributions $\mu_x$ (in red color) and $\mu_y$ (in black color).}
    \label{fig:example1}
\end{figure}

The integer-valued coupling can be viewed as a bijection between the blowup of 
$N(x)$ and $N(y)\setminus N[x]$. More discussions will be given in Section 3. It is useful for understanding the structural meaning of Ricci curvature $\kappa(x,y)$.  The special case of $d_x=d_y$ was  considered by Bai and Lei \cite{BR2022}. They derived a formula using the number of (almost)-disjoint $C_3$, $C_4$, and $C_5$ containing both $x$ and $y$. The formula can be shown to equivalent to
Equation \eqref{eq:formula}.

Lin, Lu, and Yau \cite{LLY} and Ollivier \cite{Ollivier}  proved the Bonnet-Myers type Theorem on Graphs:
\begin{lemma} \label{lem:diam}
If for every edge $xy \in E(G)$, $\kappa(x,y) \geq \kappa_1 > 0$, then the diameter of the graph $G$
$$\textrm{diam}(G) \leq \frac{2}{\kappa_1}.$$
\end{lemma}

Let $\kappa_{min}(G)=\min_{xy\in E(G)} \kappa(x,y)$.
We say a connected graph $G$ is {\em Bonnet-Myers sharp} if $diam(G)=\frac{2}{\kappa_{min}}$.  A pair of vertices $(x,y)$ in an Bonnet-Myers sharp graph is called {\em poles} if $d(x,y)=diam(G)$. In this case, we call $x$ (and $y$) is a {\em pole}.

Cushing-Kamtue-Koolen-Liu-M\"unch-Peyerimhoff\cite{CKKLMP} studied Bonnet-Myers sharp regular graphs.
They called a graph $G$ is
  {\em (D,L)-Bonnet-Myers sharp} if $G$ is $D$-regular, $diam(G)=L$, and is Bonnet-Myers sharp. A graph $G=(V,E)$ is called {\em self-centered} if, for every vertex $x \in V$, there exists a
vertex $\overline{x} \in V$ such that
$d(x,\overline{x}) = {\rm diam}(G)$. Here is their main result.

\begin{theorem}\cite{CKKLMP} \label{thm:selfcentered}
  Self-centered $(D,L)$-Bonnet-Myers sharp graphs are precisely the following
  graphs:
  \begin{enumerate}
  \item hypercubes $Q_n$, $n \ge 1$;
  \item cocktail party graphs $CP(n)$, $n \ge 3$; 
  \item the Johnson graphs $J(2n,n)$, $n \ge 3$;
  \item even-dimensional demi-cubes $Q_{2n}^{(2)}$, $n \ge 3$; 
  \item the Gosset graph;
  \end{enumerate}
  and Cartesian products of above graphs satisfying
  \begin{equation} \label{eq:cartprod_cond0} 
    \frac{D_1}{L_1} = \frac{D_2}{L_2} =\cdots = \frac{D_k}{L_k}.
  \end{equation}
\end{theorem}

Kamtue proved that a $(D,3)$-Bonnet-Myers sharp graph is automatically self-centered (see \cite{Kamtue}). Thus, the condition of being self-centered in Theorem \ref{thm:selfcentered} can be removed for graphs with diameter $3$.

Most known results assume that $G$ is $D$-regular.
The smallest irregular and not-self-centered Bonnet-Myers sharp graph is $P_3$. We have $\kappa_{min}(P_3)=1$ and
diameter $L=2$. 

In this paper, we will omit the conditions of being regular and self-centered. To the best of our knowledge, there is only one other paper that has studied Bonnet-Myers sharp irregular graphs, providing a classification for those with a diameter of $2$ and some necessary conditions for anti-trees being Bonnet-Myers sharp (see \cite{CushingStone}). They showed that a graph $G$ is a diameter $2$ Bonnet-Myers sharp graph if and only if $G$ is a complete graph with a matching removed.

The case for diameter 3 (of irregular Bonnet-Myers sharp graphs)  remains entirely unexplored.
In this paper, we establish the following structural theorem.

\begin{theorem}\label{thm:structural thm}
Suppose $G$ is a Bonnet-Myers sharp irregular graph on $n$ vertices with diameter 3, then we have:
\begin{enumerate}
    \item $G$ has a unique pair of poles, say $x,y$. We have $d_x=d_y=\frac{n-2}{2}$.

    \item  For any vertex $u \in N(y)$ (or $v\in N(x)$), 
    the induced graph $G[N(u) \cap N(x)]$ (or $G[N(v) \cap N(y)]$) is a complete graph with a matching removed, respectively.

    \item There exists two integers $r$ and $t$ satisfying $1\leq t \leq \frac{r}{2}+2$ such that $d_x=d_y=2r+t$ and 
    $d_u=3(r+1)$    
    for all other vertices $u\not =x$ or $y$.
\end{enumerate}
\end{theorem}

\begin{corollary}
    If $G$ is an irregular Bonnet-Myers sharp graph of diameter $3$, then $G$ has a unique pair of poles.
\end{corollary}
We have constructed an infinite number of Bonnet-Myers sharp irregular graphs with a diameter of 3. Specifically, there exists a Bonnet-Myers sharp graph of diameter $3$ for $t=1,2$ and all $r\geq 1$.
(We suspect that $t$ is limited to a few specific values. Determining all possible pairs of $(r,t)$ appears to be a challenging problem).
On the contrast, there are only four
Bonnet-Myers sharp regular graphs of diameter $3$: the hypercube $Q_3$, the Johnson
graph $J(6, 3)$, the demi-cubes $Q_6^{(2)}$, and the Gosset graph.

We additionally prove the following result concerning $C_3$-free Bonnet-Myers sharp graphs.
\begin{theorem}\label{thm:hypercube}
Suppose $G$ is a $C_3$-free Bonnet-Myers sharp graph with diameter $L$. If $G$ contains a pole $x$ with $d_x>\frac{2}{3}L$, then $G$ must be the hypercube $Q_L$.
\end{theorem}

When $d_x\leq\frac{2}{3}L$,  Bonnet-Myers sharp irregular graphs may exist. See examples in Table \ref{table1}.

The remainder of the paper is organized as follows: In Section 2, we establish a theorem on the existence of integer-valued optimal couplings and a lemma on complementary slackness. In Section 3, we present the proof of Theorem 4. Section 4 is dedicated to proving several lemmas for Bonnet-Myers sharp graphs. In Section 5, we prove Theorem \ref{thm:structural thm}  and also construct many diameter 3 Bonnet-Myers-sharp irregular graphs. Finally, in the last section, we examine the $C_3$-free Bonnet-Myers-sharp graphs.

\section{Notation and Lemmas on Ricci Curvature}
A {\em mass distribution} (over the vertex set $V=V(G)$) is a mapping $m: V\to [0,\infty)$. The {\em total mass} of $m$ is
$\|m \|_1=\sum_{v\in V}m(v)$. A probability distribution can be viewed as a mass distribution with a total mass of 1. The concepts of coupling 
(Equation \eqref{eq:coupling})
and transportation distance (Equations \eqref{eq:distance-coupling} and  \eqref{eq:distance-Lipschitz}) can be easily extended to two mass distributions with equal total mass.

For any positive constant $c$, let
$cm$ (or $cA$) be the distribution
(or coupling) scaled by a factor of $c$.  It is straightforward to verify
\begin{equation}\label{eq:scale}
    W(cm_1,cm_2)=cW(m_1,m_2)
\end{equation}

We first prove a lemma on the existence of integer-valued optimal coupling.

\begin{lemma}\label{lem:integerA}
If $m_1$ and $m_2$ are both integer-valued mass distributions (with finite supports), then there exists an integer-valued coupling (with finite supports) between $m_1$ and $m_2$, $A: V \times V \to \mathbb{Z}$ such that $A$ is an optimal solution of Equation \eqref{eq:distance-coupling}.
\end{lemma}

\begin{proof}
Let $A_0$ be an optimal solution of Equation \ref{eq:distance-coupling}.  Suppose $A_0$ is not integer-valued. Let $L=\{u \in V: m_1(u)>0\}$ and $R=\{v \in V: m_2(v)>0\}$. We consider a bipartite graph $H_0:=H(A_0)$ with vertex bi-partition $L \sqcup R$ (the disjoint union of $L$ and $R$, i.e., common vertices are duplicated), and the edge set 
\[E(H_0)=\{(u,v):  \mbox{and} \ A_0(u,v) \notin \mathbb{Z}\}.\]

Since $m_1$ and $m_2$ are both integer-valued, the minimum degree of $H_0$ is at least $2$ so $H_0$ contains a cycle $C$. 
Since $H_0$ is bipartite, this cycle $C$  must have even length, say $2k$. Let  $u_1,v_1,\ldots,u_k,v_k$ 
be the vertices of $C$ under a cyclic orientation. 

Here $u_1,u_2,\ldots, u_k\in L$ while $v_1,v_2,\ldots,v_k\in R$. (For convenience, let $u_{k+1}=u_1$ and $v_{k+1}=v_1$.)
Without loss of generality, we can assume
\begin{equation} \label{eq:oritation}
\sum_{i=1}^k d(u_i,v_i)\leq \sum_{i=1}^k d(u_{i+1},v_i).
\end{equation}
(Otherwise, simply reverse the orientation of $C$
and relabel its vertices.)

For $1\leq i\leq k$, define 
\begin{align*}
   \epsilon_i &=\lceil A_0(u_i,v_i)\rceil - A_0(u_i,v_i), \\
   \epsilon'_i &= A_0(u_{i+1},v_i) - \lfloor A_0(u_{i+1},v_i)\rfloor. 
\end{align*}
Here $\lceil x \rceil$ represents the ceil of $x$, which is the smallest integer that is greater than or equal to $x$. Similarly,
$\lfloor x\rfloor$ represents the floor of $x$, which is the greatest integer that is less than or equal to $x$.

Let $\epsilon_0$ be the minimum value among $\epsilon_1,\ldots, \epsilon_k, \epsilon'_1,\ldots, \epsilon'_k$.
By the definition of $H_0$, we have $\epsilon_0>0$.

Now we define a new coupling $A_1$:
\[A_1(u, v)=
\begin{cases}
    A_0(u, v)+\epsilon_0 & \mbox{if} \ (u,v)=(u_i, v_{i}) \in E(C);\\
    A_0(u, v)-\epsilon_0 & \mbox{if}\ (u,v)=(u_{i+1}, v_i) \in E(C);\\
    A_0(u, v) & \mbox{if}\ (u,v) \notin E(C).\\
\end{cases}\]
Observe that all entries of $A_1$ are still non-negative. In addition, we have
\begin{align*}
    \sum_{v}A_1(u,v)=\sum_v A_0(u,v)=m_1(u),\\
    \sum_{u}A_1(u,v)=\sum_u A_0(u,v)=m_2(v).
\end{align*}
Therefore $A_1$ is still a coupling between $m_1$ and $m_2$. Now we have
\begin{align*}
C(A_1)&=\sum_{x,y\in V} A_1(x,y) d(x,y)\\
&=\sum_{x,y\in V} A_0(x,y) d(x,y) + \epsilon_0 \sum_{i=1}^k d(u_i,v_i) - \epsilon_0 \sum_{i=1}^k d(u_{i+1},v_i)\\
&=C(A_0) + \epsilon_0 \left( \sum_{i=1}^k d(u_i,v_i) - \sum_{i=1}^k d(u_{i+1},v_i)\right)\\
&\leq C(A_0).
\end{align*}
Since $A_0$ is optimal, we must have $C(A_1)=C(A_0)$. Thus, $A_1$ is also optimal.
If some values of $A_1(u,v)$ are not integers, we can repeat this process to define a new bipartite graph $H_1=H(A_1)$. Note that $H_1$ will have at least one fewer edge than $H_0$. This process will terminate after a finite number of iterations, ultimately yielding an integer-valued optimal coupling $A$.
\end{proof}

Note that Equations \eqref{eq:curv_laplacian} and \eqref{eq:curv_coupling} are dual to each other.
We have the following lemma, which is the special case of the complementary slackness theorem in the linear programming.
\begin{lemma}\label{lem:coupling}
Suppose $f$ is an optimal solution of Equation \eqref{eq:curv_laplacian} and $B$ is an optimal solution of Equation \eqref{eq:curv_coupling}. Then
we have
\begin{enumerate}
    \item If $f(v)-f(u)\not =d(u,v)$, then $B(u,v)=0$.
    \item If $B(u,v)\not =0$, then $f(v)-f(u)=d(u,v)$.
\end{enumerate}
\end{lemma}

\begin{proof}
    We have
    \begin{align*}
        \kappa(x,y) &= \nabla_{xy} \Delta f \\
        &=\Delta f(x)-\Delta f(y)\\
        &= \sum_{u}\mu_x(u) (f(u)-f(x)) -\sum_{v}\mu_y(v) (f(v)-f(y))\\
        &=\sum_{u,v} B(u,v) (f(x)-f(u)) -  \sum_{u,v} B(u,v) (f(y)-f(v))\\
        &=\sum_{u,v} B(u,v) (f(x)-f(y)) + \sum_{u,v} B(u,v) (f(v)-f(u)) \\
        &= \sum_{u,v} B(u,v) (f(v)-f(u))\\
        &= B(x,y) + \sum_{(u,v)\not=(x,y)} B(u,v) (f(v)-f(u))\\
        &\geq B(x,y) + \sum_{(u,v)\not=(x,y)} B(u,v) d(u,v)\\
        &= \sum_{u,v} B(u,v) d(u,v)\\
        &=\kappa(x,y).    
    \end{align*}

The equality holds if and only if for any pair of vertices $(u,v)\not=(x,y)$,
$$B(u,v)(f(v)-f(u)-d(u,v))=0.$$
Thus, the assertion holds for all $(u,v)\not=(x,y)$.
On the other hand, the assertion is trivial for $(u,v)=(x,y)$. 
\end{proof}

\begin{lemma}\label{lem:duality2}
Suppose $f$ is an optimal solution of Equation \eqref{eq:curv_laplacian} and $\sigma$ is an optimal solution of Equation \eqref{eq:formula}. Then for any pair of distinct vertices
$(u,v)\not =(x,y)$,
we have
\begin{enumerate}
    \item If $f(v)-f(u)\not =d(u,v)$, then $\sigma(u,v)=0$.
    \item If $\sigma(u,v)\not =0$, then $f(v)-f(u)=d(u,v)$.
\end{enumerate}
\end{lemma}
\begin{proof}

We define a $\ast$-coupling $B_\sigma$ between $m_x^0$ and $m_y^0$ as follows.
For any $u,v\in V(G)$, we have
\begin{equation} \label{equ:Bcoupling}
B_\sigma(u,v)=
\begin{cases}
1+\frac{1}{d_y} & \mbox{ if } (u,v)=(x,y);\\
-\frac{1}{d_y} & \mbox{ if } u=v \in N[x] \cap N[y];\\
- \frac{\sigma(u,v)}{\lcm(d_x,d_y)} & \mbox{ if } (u,v)\in N[x] \times (N(y) \setminus N[x]);\\
0 & \mbox{ otherwise.}
\end{cases}
\end{equation}

Now we verify that $B_\sigma$ is a $\ast$-coupling between $m^0_x$ and $m^0_y$ by checking the four conditions in its definition. Condition (1) is satisfied by the definition.
For Condition (2), we have
\begin{align*}
   \sum\limits_{u,v\in V(G)} B(u,v) &= 1+\frac{1}{d_y} -\frac{|N[x]\cap N[y]|}{d_y}
   -\frac{\sum_{u,v}\sigma(u,v)}{\lcm(d_x,d_y)}\\
   &= 1+\frac{1}{d_y} -\frac{|N(x)\cap N(y)|+2}{d_y}
   -\frac{\|\mu_y\|_1}{\lcm(d_x,d_y)}\\
   &=1+\frac{1}{d_y} -\frac{|N(x)\cap N(y)|+2}{d_y}
   -\frac{(d_y-1-|N(x)\cap N(y)|)c_y}{\lcm(d_x,d_y)}\\
   &=0.
\end{align*}
Both sides of the equation in Condition (3) are zeros unless 
$u\in N[x]$. Since $u\not=x$,  we may assume $u\in N(x)$.
We have
\begin{align*}
    \sum\limits_{v \in V(G)} B(u, v)
    &= -\frac{\mathbf{1}_{u\in N(x)\cap N[y]}}{d_y} -
    \frac{\sum\limits_{v \in N(y)\setminus N[x] } \sigma(u,v)}{\lcm(d_x,d_y)}\\
    &= -\frac{c_y \mathbf{1}_{u\in N(x)\cap N[y]}}{\lcm(d_x,d_y)} - \frac{\mu_x(u)}{\lcm(d_x,d_y)}\\
    &= -\frac{c_x}{\lcm(d_x,d_y)}\\
    &=-\frac{1}{d_x}\\
    &=-m_x^0(u).
\end{align*}
Now we verify Condition (4). Similarly, we may assume $v\in N(y)$. We divide it into two cases. Case 1: $v\in N(y)\cap N[x]$. We have
\begin{align*}
    \sum\limits_{u \in V(G)} B(u, v)
    &= - \frac{1}{d_y}\\
    &=-m^0_y(v).
\end{align*}
Case 2: $v\in N(y)\setminus N[x]$. We have
\begin{align*}
    \sum\limits_{u \in V(G)} B(u, v)
    &= -    \frac{\sum\limits_{u \in N[x] } \sigma(u,v)}{\lcm(d_x,d_y)}\\
    &=  - \frac{\mu_y(v)}{\lcm(d_x,d_y)}\\
    &= -\frac{c_y}{\lcm(d_x,d_y)}\\
    &=-\frac{1}{d_y}\\
    &=-m_y^0(u).
\end{align*}
All four conditions in the definition of $\ast$-coupling are verified.

From our construction, we observe that for $u\not =v$ and $(u,v)\not=(x,y)$ we have $\sigma(u,v)>0$ if and only if $B_\sigma(u,v)\not =0$. The conclusion follows from Lemma \ref{lem:coupling}.
\end{proof}

\section{Proof of Theorem \ref{thm:localstructure}}
Before we prove Theorem 
\ref{thm:localstructure}, we would like to give an structural interpretation.

Assume $d_x\leq d_y$. For any $u\in N(x)\cup N(y)$, we define a blowup set $S_u$ (depending on the vertex $u$) as follows:
\[S_u=
\begin{cases}
\emptyset &\mbox{ if } u=x;\\
c_x-c_y \mbox { copies of } u & \mbox{ if }  u\in N(x)\cap N[y]\\
c_x \mbox { copies of } u & \mbox{ if } u\in N(x)\setminus N[y];\\
c_y \mbox { copies of } u & \mbox{ if } u\in N(y)\setminus N[x].
\end{cases}
\]
Define two multi-sets $X:=X(x,y)=\cup_{u\in N(x)} S_u$ and  $Y:=Y(x,y)=\cup_{u\in N(y)\setminus N[x]}S_u$. 
Then we construct an auxiliary complete bipartite edge-weighed graph $H:=H(x,y)$ with vertex-partition $X\cup Y$.
The weight of the edge $(u_i, v_j)\in X\times Y$ is set to $d(u_i,v_j)=d(u,v)$ if $u_i\in S_u$ and $v_j\in S_v$.

We have
\begin{align*}
|X|&=(c_x-c_y)(|N(x)\cap N(y)|+1) + c_x |N(x)\setminus N[y]|\\
&= c_x (|N(x)\cap N(y)|+1 + |N(x)\setminus N[y]|) -c_y (|N(x)\cap N(y)|+1)\\
&=c_x d_x -c_y (|N(x)\cap N(y)|+1)\\
&=c_y d_y -c_y (|N(x)\cap N(y)|+1)\\
&= c_y(d_y- |N(x)\cap N(y)|-1)\\
&=c_y |N(y)\setminus N[x]|\\
&=|Y|.
\end{align*}

Any integer-valued coupling $\sigma$ between $\mu_x$ and $\mu_y$ can be viewed as a perfect matching in $H$.
The cost $C(\sigma)$ is the sum of the edge-weights of the perfect matching.

Let $H_1:=H_1(x,y)$ be a bipartite graph on $(X,Y)$ whose edge set consisting of all edges with weight $1$ in $H$.
The graph $H_1$ can be viewed as a blow-up graph of the bipartite induced graph $G[N(x), N(y)\setminus N[x]]$.
We have the following corollary.

\begin{corollary}\label{cor:k upper bound}
    For any edge $xy\in E(G)$, assume $d_x\leq d_y$. We have
    \[\kappa(x,y)\leq \frac{|N(x)\cap N(y)|+2}{d_y}.\]
The equality holds if and only if there is a perfect matching in $H_1(x,y)$. 
\end{corollary}

\begin{proof}
\begin{align*}
    \kappa(x,y) & = 1+\frac{1}{d_y}-\frac{\min_{\sigma} C(\sigma)}{\lcm(d_x,d_y)}\\
    & \leq 1+\frac{1}{d_y}-\frac{\|\mu_y\|_1}{\lcm(d_x,d_y)}\\
    & = 1+\frac{1}{d_y}-\frac{c_y(d_y-1-|N(x) \cap N(y)|)}{\lcm(d_x,d_y)}\\
    &=1+\frac{1}{d_y}-\frac{d_y-1-|N(x) \cap N(y)|}{d_y}\\
     &=\frac{2+|N(x) \cap N(y)|}{d_y}.
\end{align*}    
If the equality holds, then $\min_{\sigma} C(\sigma)=\|\mu_y\|_1$. That is, $d(u,v)=1$ whenever $\sigma(u,v)>0$. Equivalently, there is a perfect matching in $H_1(x,y)$.
\end{proof}

We have the following corollary of the Hall theorem.
\begin{corollary}\label{cor:Hall}
The graph $H_1(x,y)$ has a perfect matching if and only if
\begin{equation}\label{eq:hallcondition}
  c_y|N(T_1\cup T_2) \cap (N(y)\setminus N[x])|\geq c_x|T_1|+(c_x-c_y)|T_2|,  
\end{equation}
holds (in $G$) for any
 $T_1 \subseteq N(x) \setminus N[y]$ and $T_2 \subseteq N(x)\cap N[y]$. 
\end{corollary}

\begin{proof}
On one direction, if $H_1$ has a perfect matching, then
$|N_{H_1}(K)|\geq |K|$, for any $K\subseteq X$.
In particular, it holds for $K=\cup_{u\in T_1\cup T_2}S_u$. We have
\[
c_x|T_1|+(c_x-c_y)|T_2|=|K|\leq |N_{H_1}(K)|=
c_y|N(T_1\cup T_2) \cap (N(y)\setminus N[x])|.
\]

On the other direction, suppose Inequality \eqref{eq:hallcondition} holds for any $T_1$ and $T_2$.
We will verify the Hall's condition for $H_1$.
For any $K\subseteq X$, define
$T=\{u\in N(x)\colon K\cap S_u\not=\emptyset\}$.
Let $T_1=T\cap  (N(x) \setminus N[y])$ and
$T_2=T\cap (N(x)\cap N[y])$. Let $\tilde T=\cup_{u\in T}S_u$ and $\widetilde {N(T)}= \cup_{u\in N(T)\cap 
(N(y)\setminus N[x])} S_u$.

Observe $N_{H_1} (\tilde T)=N_{H_1}(K)= \widetilde {N(T)}$ and
$K\subseteq \tilde T$.
We have
\begin{align*}
 |N_{H_1}(K)|&=|N_{H_1}(\tilde T)|\\
&=|\widetilde{N(T)}|\\   
&= c_y|N(T_1\cup T_2) \cap (N(y)\setminus N[x])|\\
&\geq c_x|T_1|+(c_x-c_y)|T_2|\\
&=|\tilde T|\\
&\geq |K|.
\end{align*}
By the Hall theorem, $H_1$ has a perfect matching.
\end{proof}

Now we are ready to prove Theorem \ref{thm:localstructure}.

\begin{proof}[Proof of Theorem \ref{thm:localstructure}]
 Applying Equation \eqref{eq:BCL} with $\alpha=\frac{1}{d_y+1}$, we have
 \begin{align}
     \kappa(x,y)&=\frac{\kappa_\alpha}{1-\alpha}  \nonumber \\
     &= \frac{1}{1-\alpha}-\frac{1}{1-\alpha}W(m_x^\alpha, m_y^\alpha) \nonumber \\
     &=1+\frac{1}{d_y} -\frac{1}{\lcm(d_x,d_y)} W\left(\frac{\lcm(d_x,d_y)}{(1-\alpha)} m_x^\alpha,  \frac{\lcm(d_x,d_y)}{(1-\alpha)}m_y^\alpha\right) \nonumber \\
     &=1+\frac{1}{d_y} -\frac{1}{\lcm(d_x,d_y)} W( \tilde \mu_x, \tilde \mu_y).
     \label{eq:muxy}
 \end{align}
Here $\tilde \mu_x$ and $\tilde \mu_y$ are short notations for $\frac{\lcm(d_x,d_y)}{(1-\alpha)} m_x^\alpha$ and  $\frac{\lcm(d_x,d_y)}{(1-\alpha)} m_y^\alpha$, respectively.
It is easy to calculate
\begin{align}
  \tilde \mu_x(u)=
\begin{cases}
  c_y & \mbox{ if } u=x,\\
  c_x & \mbox{ if } u\in N(x), \\
  0 & \mbox{ otherwise;}
\end{cases} \\
\tilde \mu_y(u)=
\begin{cases}
  c_y & \mbox{ if } u=y,\\
  c_y & \mbox{ if } u\in N(y), \\
  0 & \mbox{ otherwise.}
\end{cases}
\end{align}

Now we prove $W(\tilde \mu_x, \tilde \mu_y)= W(\mu_x,\mu_y)$.

First we show $W(\tilde \mu_x, \tilde \mu_y)\leq  W(\mu_x,\mu_y)$.
Suppose that $\sigma$ is an optimal coupling between $\mu_x$ and $\mu_y$, i.e.,
$W(\mu_x,\mu_y)=C(\sigma)$.

Now define a map $\tilde \sigma
\colon V\times V\to [0,\infty)$ as follows.
\[
\tilde \sigma(u,v) = \begin{cases}
\sigma(u,v) + c_y & \mbox{ if } u=v\in N[x]\cap N[y];\\
   \sigma(u,v) & \mbox{ otherwise}.
\end{cases}
\]
It is straightforward to verify $\tilde \sigma$ is indeed a coupling between
$\tilde \mu_x$ and $\tilde \mu_y$. We have
\[C(\tilde \sigma)= C(\sigma) + \sum_{u\in N[x]\cap N[y]} c_y d(u,u)= C(\sigma).\]
Thus,
\[ W(\tilde \mu_x, \tilde \mu_y)\leq  C(\tilde \sigma) =C(\sigma) =W(\mu_x, \mu_y).\]

Next, we will show $W(\tilde \mu_x, \tilde \mu_y)\geq  W(\mu_x,\mu_y)$. 
Since both $\tilde \mu_x$ and $\tilde \mu_y$ are integer-valued, by Lemma \ref{lem:integerA},
there exists an integer-valued optimal coupling $\tilde \sigma_0$ between $\tilde \mu_x$ and $\tilde \mu_y$. 
We would like to construct a new integer-valued optimal coupling $\tilde \sigma$ satisfying the following property. 

{\bf Property A:}  For any $z\in N[x]\cap N[y]$, $\tilde \sigma(z,z)=c_y$. 

Suppose $\tilde \sigma_0$ does not have Property A. There is a $z_0\in N[x]\cap N[y]$ such that $\tilde \sigma_0(z_0,z_0)\not=c_y$. Since $\tilde \sigma_0(z,z)\leq 
\tilde \mu_y(z)=c_y$, we must have 
$$\tilde \sigma_0(z_0,z_0) <c_y.$$
There exists a pair of vertices $u_0$ and $v_0$ with $\tilde \sigma_0(u_0, z_0)>0$
and $\tilde \sigma_0(z_0,v_0)>0$.

Let $c=\min\{ \tilde \sigma_0(u_0, z_0), \tilde \sigma_0(z_0,v_0)\}$.
Now define a mapping $\tilde \sigma_1: V\times V\to [0,\infty)$ as follows.

\[ \tilde \sigma_1(u,v) = \begin{cases}
  \tilde\sigma_0(u_0,z_0) -c
 & \mbox{ if } (u,v)=(u_0,z_0);\\
  \tilde\sigma_0(z_0,v_0) -c
 & \mbox{ if } (u,v)=(z_0,v_0);\\
  \tilde\sigma_0(z_0,z_0) +c 
 & \mbox{ if } u=v=z_0;\\
   \tilde\sigma_0(u_0,v_0) +c 
 & \mbox{ if } (u,v)=(u_0,v_0);\\
  \tilde \sigma_0(u,v) & \mbox{ otherwise}.
\end{cases}
\]

It is straightforward to verify $\tilde \sigma_1$ is still a coupling between
$\tilde \mu_x$ and $\tilde \mu_y$. Since $\tilde \sigma_0$ is integer-valued, $\tilde \sigma_0(u_0, z_0)$, $\tilde \sigma_0(z_0,v_0)$, and $c$ are integers.
Thus, $\tilde \sigma_1$ is integer-valued. Furthermore,
we have
\[C(\tilde \sigma_1) =C(\tilde \sigma_0) -c\left(d(u_0,z_0)+d(z_0,v_0)-d(u_0,v_0)\right)
\leq C(\tilde \sigma_0). 
\]
Since $\tilde \sigma_0$ is optimal, we must have
$ C(\tilde \sigma_1)=C(\tilde \sigma_0)$. Thus, $\tilde \sigma_1$ is also optimal.

We can iterate this process to construct a series of integer-valued optimal coupling
$\tilde \sigma_1$, $\tilde \sigma_2$, $\ldots$,
until Property A is satisfied.  Let 
$f(\tilde \sigma_i)=\sum_{z\in N[x]\cap N[y]} \tilde \sigma_i(z,z)$. Observe that
$f(\tilde \sigma_i)> f(\tilde \sigma_{i-1})$ during the process. The process must terminate in a finite number of steps because $f$ is an integer-valued increasing function and is bounded by $c_y |N[x]\cap N[y]|$.

Upon termination of the process, we obtain an integer-valued optimal coupling $\tilde \sigma$ satisfying Property A. Now we define a map $\sigma\colon V\times V\to [0,\infty)$
as follows.

\[\sigma(u,v) =
\begin{cases}
    0 & \mbox{ if } u=v\in N[x]\cap N[y];\\
  \tilde \sigma(u,v) & \mbox{ otherwise.}   
\end{cases}
\]
It is easy to confirm that $\sigma(u,v)$ is an integer-valued coupling between
$\mu_x$ and $\mu_y$.  Note that altering the values on the diagonal does not affect the cost.
We have
\[C(\tilde \sigma_0)=\cdots=C(\tilde \sigma)=C(\sigma).\]
Therefore,
\[W(\mu_x,\mu_y)\leq C(\sigma)=C(\tilde \sigma_0)=W(\tilde \mu_x, \tilde \mu_y). \]

Thus, we proved
\[ W(\tilde \mu_x, \tilde \mu_y)= W(\mu_x,\mu_y).\]
Plugging to Equation \eqref{eq:muxy}, we get

\[ \kappa(x,y)= 
1+\frac{1}{d_y}- \frac{W(\mu_x,\,u_y)}{{\lcm(d_x, d_y)}}
=
1+\frac{1}{d_y}-\frac{\min_\sigma C(\sigma)} {\lcm(d_x, d_y)}.\]
Here the minimum is achieved by some integer-valued coupling $\sigma$ between
$\mu_x$ and $\mu_y$.
\end{proof}

\section{Notation and Lemmas on Bonnet-Myers sharp graphs}

From now on, we assume $G=(V,E)$ is a simple connected finite graph so it has a finite diameter.
We always assume $diam(G)=L$ unless $L$ being specified. For any two vertices $u,v \in V(G)$, a {\em geodesic} from $u$ to $v$ is a shortest path from $u$ to $v$. 
We use an interval $[u,v]$ to denote the set of all vertices lying on geodesics from $u$ to $v$:
\[[u,v]=\{w \in V(G): d(u,w)+d(w,v)=d(u,v)\}.\]

Fix a pair of poles $(x,y)$ of $G$. We use $N_i(x)=\{u\in V(G): d(x,u)=i\}$ to denote the $i$-th neighbor of $x$. For any $1$-Lipschitz function $f\colon V\to \mathbb Z$ with $f(x)=0$ and $f(y)=L$, we define 
\begin{align*}
    N_f^+(u)&=\{v\in N(u)\colon f(v)=f(u)+1\};\\
    N_f^0(u)&=\{v\in N(u)\colon f(v)=f(u)\};\\
    N_f^-(u)&=\{v\in N(u)\colon f(v)=f(u)-1\},
\end{align*}
If $f$ is clear under context, we will drop the subscript $f$.

For this section, we assume
$G$ is a Bonnet-Myers sharp graph with diameter $L$, and $x,y$ is a pair of poles.
\begin{lemma} \label{lem:xi}
 Let $f\colon V\to \mathbb Z$ be a $1$-Lipschitz function with $f(x)=0$ and $f(y)=L$. For any shortest $xy$-path $x=x_0,x_1,\ldots, x_{L-1},x_L=y$, we have
\begin{enumerate}
    \item For $i=1,2,\ldots, L$, we have $\kappa(x_{i-1},x_{i})=\frac{2}{L}$. Moreover, $f$ is an optimal solution of Equation \eqref{eq:curv_laplacian} for computing $\kappa(x_{i-1},x_i).$
    \item We have   
      \begin{equation} \label{eq:17}
            |N_f^+(x_i)| - |N^-_f(x_i)|=(1- \frac{2i}{L}) d_{x_i},
        \end{equation}
for $0\leq i \leq L$. In particular, $2id_{x_i}$ is divisible by $L$.
\end{enumerate}
\end{lemma}
\begin{proof} 
   We have
    \begin{align*}
         2=L\kappa_{min} &\leq \sum_{i=1}^L\kappa(x_{i-1},x_i)\\
         &\leq \sum_{i=1}^L (\Delta f(x_{i-1}) -\Delta f(x_i))\\
         &= \Delta f(x_0) -\Delta f(x_L)\\
         &\leq 1 -(-1)\\
         &=2.
    \end{align*}
Therefore, all of the inequalities occurred above are indeed equalities. Specifically, for $1\leq i\leq L$, we have
$$\frac{2}{L}=\kappa_{min}=\kappa(x_{i-1},x_i)=\Delta f(x_{i-1}) -\Delta f(x_i).$$
Thus, $f$ is an optimal solution of Equation \eqref{eq:curv_laplacian} for computing $\kappa(x_{i-1},x_i)$.

Since $\Delta f(x_0)=1$, we have
$$\Delta f(x_i) = 1-\frac{2i}{L}, \mbox { for } 0\leq i \leq L.$$
Equation \eqref{eq:17} follows from the fact that $$\Delta f(x_i)=\frac{|N_f^+(x_i)|- |N_f^-(x_i)|}{d_{x_i}}.$$
\end{proof}

Lemma \ref{lem:xi} leads directly to the next lemma, which provides a formula for the ratio of $|N_f^+(u)|$ and $|N_f^-(u)|$ for a vertex $u$ when $|N_f^0(u)|=0$.  This lemma states that the ratio depends solely on the distance between $u$ and $x$.

\begin{lemma} \label{lem:biratio}
 Let $f\colon V\to \mathbb Z$ be a $1$-Lipschitz function with $f(x)=0$ and $f(y)=L$.  Then for $u \in N_i(x)$ with $|N_f^0(u)|=0$, we have
\[\frac{|N_f^+(u)|}{|N_f^-(u)|}=\frac{L}{i}-1,\]
for $i=1,\ldots,L$.
\end{lemma}

\begin{proof}
By Lemma \ref{lem:xi}, we have 
\[|N_f^+(u)|-|N_f^-(u)|=\left(1-\frac{2}{L}i\right)d_{u}.\]    
Since $d_{u}=|N_f^+(u)|+|N_f^-(u)|+|N_f^0(u)|=|N_f^+(u)|+|N_f^-(u)|$, we have
\[|N_f^+(u)|-|N_f^-(u)|=\left(1-\frac{2}{L}i\right)\left(|N_f^+(u)|+|N_f^-(u)| \right).\] 
It implies
\[\frac{|N_f^+(u)|}{|N_f^-(u)|}=\frac{L}{i}-1.\]
\end{proof}

\begin{lemma}\label{lem:interval}
     We have $[x,y]=V(G)$.
\end{lemma}

\begin{proof}
   We will prove by contradiction. Suppose there are some vertex in $V\setminus [x,y]$. Let $C$ be a connected component of the induced subgraph of $G$ on $V\setminus [x,y]$. Let $f(v)=d(x,v)$ be the distance function from $x$ to $v$. Clearly $f$ is a $1$-Lipschitz function with $f(x)=0$ and $f(y)=L$.
    Now we define a new function $\tilde f$:
    $$\tilde f(v)=\begin{cases}
        f(v) -1 & \mbox{ if } v\in V(C),\\
        f(v) & \mbox{otherwise}.
    \end{cases}$$

    We claim that $\tilde f$ is also a $1$-Lipschitz function with $f(x)=0$ and $f(y)=L$. If not, there exists a pair of vertex $(u,v)$ such that 
    \[|\tilde f(u) -\tilde f(v)|>d(u,v).\]
    This can only occur when one vertex is in $C$ and the other vertex is in $[x,y]$, say $v\in V(C)$ and $u\in [x,y]$. It is also necessary to have $f(u)>f(v)$ and
    \[\tilde f(u) -\tilde f(v) \geq d(u,v)+1.\]
This implies $f(u)-f(v)=d(u,v)$.
Since $u\in [x,y]$, we have $d(u,y)=f(y)-f(u)$. 
Hence,
$$d(x,v)+ d(v,u) +d(u,y)=f(v)-f(x) + f(u)-f(v) + f(y)-f(u)=f(y)-f(x)=d(x,y).$$
Thus $x-v-u-y$ is a shortest path from $x$ to $y$. Contradiction!

Consider $v_0$ as a vertex in $C$ that
attains the minimum value of $f$, say $f(v_0)=i+1$ for some $i$.
Along a shortest path from $x$ to $v_0$, there exists a vertex, denoted as $x_{i}$, which is located within the interval $[x,y]$ and satisfies the condition $f(x_{i})=i$.
Let $x=x_0, x_1, \ldots, x_{i},\ldots, x_L=y$ be a shortest $xy$-path passing through $x_{i}$. 
Applying Lemma \ref{lem:xi} to $x_i$ with respect to $f$, we have
\begin{equation}\label{eq:6}
 |N_f^+(x_i)| - |N^-_f(x_i)|=(1- \frac{2i}{L}) d_{x_i}.
\end{equation}
Now Applying Lemma \ref{lem:xi} to $x_i$ with respect to $\tilde f$, we have
\begin{equation} \label{eq:7}
|N_{\tilde f}^+(x_i)| - |N^-_{\tilde f}(x_i)|=(1- \frac{2i}{L}) d_{x_i}.
\end{equation}
By the choice of $x_i$, we have
\begin{equation}\label{eq:8}
|N^-_{\tilde f}(x_i)|=|N^-_f(x_i)| \mbox{ but }  |N^+_{\tilde f}(x_i)|<|N^+_f(x_i)|.
\end{equation}
Equations \eqref{eq:6}, \eqref{eq:7}, and \eqref{eq:8} conflict to one another. Contradiction!

Therefore $[x,y]=V(G)$.
\end{proof}

This implies that for any pole $x$, there is a unique vertex $y$, such that $d(x,y)=L$. (Otherwise, there is another vertex, say $y'$,  satisfying $d(x,y')=L$. Then $y'$ is not in the interval $[x,y]$. Contradiction!) We call $y$ is the {\em anti-pole} of $x$. 

\begin{corollary}\label{cor:4}
The distance function from $x$, i.e., $f(z):=d(x,z)$ (for $z \in V(G)$), is the unique $1$-Lipschitz function satisfying $f(x)=0$ and $f(y)=L$.
\end{corollary}

From now on, unless specified otherwise, $f$ is always referred to the distance function $d(x, \bullet)$. 

\begin{corollary}\label{cor:matchingdirection}
   Consider an edge $uv$ with $d(x,u)<d(x,v)$.
   Let $\sigma$ be any optimal coupling of Equation \eqref{eq:formula} for computing $\kappa(u,v)$. For any $w\in X(u,v)$ and $z\in Y(u,v)$ with $\sigma(w,z)>0$, 
   we have
   \[d(x,w)<d(x,z).\]
\end{corollary}
\begin{proof}
    By Lemma \ref{lem:xi}, $d(x,\bullet)$ is an optimal function of Equation \eqref{eq:curv_laplacian}. Then applying Lemma \ref{lem:duality2}, we have
    \[d(x,z)-d(x,w)=d(z,w)>0.\]
\end{proof}

We typically draw the graph $G$ such that the function $d(x, \bullet)$ increases from left to right. For simplicity, we say that $\sigma$ transfers masses from left to right.

We define $d^+_u = |N_f^+(u)|$, $d_u^0 =| N_f^0(u)|$, and $d_u^-=|N_f^-(u)|$.
\begin{lemma} \label{lem:curvature}
 Let $u,v$ be two vertices on a shortest $xy$-path, i.e. $d(x,y)=d(x,u)+d(u,v)+d(v,y)$. 
Then, $\kappa(u,v)=\frac{2}{L}$. 
\end{lemma}

\begin{proof}
By definition, we have 
\begin{align*}
\kappa(u,v) &\leq \frac{1}{d(u,v)}\left(\Delta f(u)-\Delta f(v)\right)\\
&=\frac{1}{d(u,v)}\left(\left(1-\frac{2}{L}d(x,u)\right)-\left(1-\frac{2}{L}d(x,v)\right)\right)\\
&=\frac{1}{d(u,v)}\left(\frac{2}{L}\left(d(x,v)-d(x,u)\right)\right)\\
&=\frac{1}{d(u,v)}\left(\frac{2}{L}d(u,v)\right)\\
&=\frac{2}{L}.
\end{align*}    
The second line comes from Lemma \ref{lem:xi} since any vertex in $G$ lies on a geodesic from $x$ to $y$. Since $\kappa(u,v) \geq \kappa_{min}=\frac{2}{L}$ by Lemma \ref{lem:adj-pair}, we have $\kappa(u,v)=\frac{2}{L}$. 
\end{proof}

{\bf Remark:} Cushing-Kamtue-Koolen-Liu-M\"unch-Peyerimhoff  \cite{CKKLMP} proved similar results to our Lemma \ref{lem:xi}, \ref{lem:interval}, and \ref{lem:curvature} on D-regular Bonnet-Myers sharp graphs. See Theorem 5.4, Theorem 5.5, and Lemma 5.3 in \cite{CKKLMP} respectively.

The next lemma gives a degree upper bound for the poles in a Bonnet-Myers sharp graph.
\begin{lemma}\label{lem:pole}
 For any vertex $u \in N(x)$, we have $d_x \leq d_u$.
\end{lemma}

\begin{proof}
 Take any $u \in N(x)$. Now we define a new function $\tilde f$:
    $$\tilde f(z)=\begin{cases}
       f(z)  & \mbox{ if } z=x \ \mbox{or} \ u,\\
       f(z)-1  & \mbox{otherwise}.
    \end{cases}$$
Clearly, $\tilde f$ is a $1$-Lipschitz function. We have 
\begin{align*}
\Delta \tilde f(x)-\Delta f(x)&=-\frac{d_x-1}{d_x},\\
\Delta \tilde f(u)-\Delta f(u)&=-\frac{d_u-1}{d_u}.
\end{align*}

From the proof of Lemma \ref{lem:xi}, we know that $f$ is an optimal solution of equation \eqref{eq:curv_laplacian}. Thus,

   \begin{align*}
        \kappa(x,u) &= \nabla_{xu} \Delta f \\
        &=\Delta f(x)-\Delta f(u)\\
        &=\Delta \tilde f(x)-\Delta \tilde f(u)+\frac{1}{d_u}-\frac{1}{d_x}.\\
    \end{align*}

By the definition of Lin-Lu-Yau curvature, we have $\kappa(x,u) \leq \Delta \tilde f(x)-\Delta \tilde f(u)$. Hence, 
\begin{align*}
\frac{1}{d_u}-\frac{1}{d_x} \leq 0,\\
d_x \leq d_u.
\end{align*}
\end{proof}

\begin{lemma} \label{lem:d upper bound}
 For any edge $uv$, we have $$d_u \leq \left(\frac{|N(u) \cap N(v)|+2}{2}\right)L.$$ 
 Moreover, if $d_u = \left(\frac{|N(u) \cap N(v)|+2}{2}\right)L$ for some $v \in N(u)$ with $d_v \leq d_u$, then there is a perfect matching in $H_1(v,u)$.
\end{lemma}

\begin{proof}
Let $v \in N(u)$. We can assume $d_v \leq d_u$ since otherwise we can let $u \in N(v)$ and use the vertex $v$ instead. By Corollary \ref{cor:k upper bound}, we have 
$$\kappa(v,u)\leq \frac{|N(u)\cap N(v)|+2}{d_u}.$$
Since $\kappa(v,u) \geq \kappa_{min}=\frac{2}{L}$, we have 
$$d_u \leq \left(\frac{|N(u) \cap N(v)|+2}{2}\right)L.$$
It is easy to see that when $d_u = \left(\frac{|N(u) \cap N(v)|+2}{2}\right)L$, the upper bound of $\kappa(v,u)$ is achieved and equals to $\frac{2}{L}$. The conclusion follows from Corollary \ref{cor:k upper bound}. 
\end{proof}

\begin{lemma}\label{lem:perfect matching}
For any $u \in N(x)$, there is a perfect matching in $H_1(x,u)$.
\end{lemma}

\begin{proof}
We have $d_x \leq d_u$ by Lemma \ref{lem:pole}. From Lemma \ref{lem:xi}, we have $d^+_u - d^-_u=\left(1- \frac{2}{L}\right) d_{u}$. Since $d^-_u=1$, we have $d^+_u=\left(1- \frac{2}{L}\right) d_{u}+1$. Then, we have \[|N(x) \cap N(u)|=d^0_u=d_u-d^-_u-d^+_u=\frac{2}{L}d_u-2.\]

By Corollary \ref{cor:k upper bound}, we have 
\begin{align*}
\kappa(x,u)& \leq \frac{|N(x) \cap N(u)|+2}{d_u}\\
&=\frac{\frac{2}{L}d_u-2+2}{d_u}\\
&=\frac{2}{L}.
\end{align*}
On the other hand, since $u$ lies on a geodesic from $x$ to $y$, we have $\kappa(x,u)=\frac{2}{L}$ by Lemma \ref{lem:curvature}. Hence, the conclusion follows from Corollary \ref{cor:k upper bound} since the upper bound is achieved. 
\end{proof}

\begin{lemma}\label{lem:private neighbor}
 For any vertex $v \in N_2(x)$, we have $|(N(x) \cap N(v)) \setminus N[u]| \leq \frac{d_v}{d_u}$ for any $u \in N(x) \cap N(v)$.
\end{lemma}

\begin{proof}
 We define a new function $\tilde f$:
    $$\tilde f(z)=\begin{cases}
       f(z)+1  & \mbox{ if } z=x \ \mbox{or} \ z \in (N(x) \cap N(v)) \setminus N[u],\\
       f(z)  & \mbox{otherwise}.
    \end{cases}$$    
Clearly $\tilde f$ is a $1$-Lipschitz function. We have
\begin{align*}
\Delta \tilde f(u)-\Delta f(u)&=\frac{1}{d_u},\\
\Delta \tilde f(v)-\Delta f(v)&=\frac{|(N(x) \cap N(v)) \setminus N[u]|}{d_v}.
\end{align*}
Taking a difference of the two equations,  we have
\[(\Delta \tilde f(u)-\Delta \tilde f(v))-(\Delta f(u)-\Delta f(v))=\frac{1}{d_u}-\frac{|(N(x) \cap N(v)) \setminus N[u]|}{d_v}.\]

By definition, $\kappa(u,v) \leq \Delta \tilde f(u)-\Delta \tilde f(v)$. Since $f$ is the optimal solution of equation \eqref{eq:curv_laplacian}, $\kappa(u,v)=\Delta f(u)-\Delta f(v)$. We have
\[\frac{1}{d_u}-\frac{|(N(x) \cap N(v)) \setminus N[u]|}{d_v} \geq 0.\]
It implies
\[|(N(x) \cap N(v)) \setminus N[u]| \leq \frac{d_v}{d_u}.\]

\end{proof}

\subsection{Lichnerowicz sharp graphs}
The Ricci curvature is known to be related to the first non-zero eigenvalue, $\lambda_1$, of the normalized combinatorial Laplacian $L=I -D^{-\frac{1}{2}}AD^{-\frac{1}{2}}$. Here $D$ represents the diagonal matrix of degrees, while $A$ denotes the adjacency matrix. It is important to note that $L$ and $-\Delta$ are similar matrices, sharing the same set of eigenvalues.

\begin{lemma}[Lichnerowicz type Theorem on graphs]
\cite{LLY, Ollivier}\label{lem:lambda1}
If for every edge $xy \in E(G)$, $\kappa(x,y) \geq \kappa_1 > 0$, then we have
$$\lambda_1(G)\geq \kappa_1.$$
\end{lemma}

We say a connected graph $G$ is {\em Lichnerowicz sharp} if $\lambda_1(G)=\kappa_{min}$. 
Cushing-Kamtue-Koolen-Liu-M\"unch-Peyerimhoff  \cite{CKKLMP} 
proved the following theorem.
\begin{theorem} \cite{CKKLMP} 
\label{thm:DLLich}
  Every $(D,L)$-Bonnet-Myers sharp graph $G = (V,E)$ is Lichnerowicz
  sharp with Laplace eigen-function $g(\bullet) = d(x,\bullet) - \frac{L}{2}$,
  where $x$ is a pole of $G$.
\end{theorem}

We extend their result to all Bonnet-Myers sharp graphs.
\begin{theorem} \label{thm:Lich}
  Every Bonnet-Myers sharp graph $G = (V,E)$ is Lichnerowicz
  sharp with Laplace eigen-function $g(\bullet) = d(x,\bullet) - \frac{L}{2}$,
  where $x$ is a pole of $G$.
\end{theorem}
\begin{proof}
Since $\lambda_1\geq \kappa_{min}=\frac{2}{L}$, it suffices to show $\frac{2}{L}$ is an eigenvalue for $-\Delta$.

Let $f(\bullet)=d(x, \bullet)$ denote the distance function to $x$.
Applying Lemma \ref{lem:xi} and Corollary \ref{cor:4},
we have
\begin{align*}
    \Delta g(u) &= \frac{1}{d_u} \sum_{v\in N(u)} (g(v)-g(u))\\
    &= \frac{1}{d_u} \sum_{v\in N(u)} (f(v)-f(u))\\
    &=\frac{1}{d_u} (d^+_u -d^-_u)\\
    &=1-\frac{2d(x,u)}{L}\\
    &=1- \frac{2(g(u)+\frac{L}{2})}{L}\\
    &=-\frac{2}{L}g(u).
\end{align*}
Therefore $g$ is the eigenvector of the Laplace $\Delta$ with eigenvalue $-\kappa_{min}$.

   \end{proof}

\section{Bonnet-Myers sharp graphs of diameter 3}

Cushing-Kamtue-Koolen-Liu-M\"unch-Peyerimhoff completely classified self-centered Bonnet-Myers sharp graphs of diameter $3$ in \cite{CKKLMP}. There are only four graphs:  the hypercube $Q_3$, the Johnson graph $J(6,3)$, the demi-cubes $Q_{(2)}^{6}$, and the Gosset graph. In this section, we discard the assumptions of self-centeredness and regularity, and outline the necessary conditions for a general diameter $3$ graph to be Bonnet-Myers sharp. We will show that an irregular Bonnet-Myers sharp graph of diameter 3 has a single unique pair of poles; otherwise, it is regular. Kamtue demonstrated in \cite{Kamtue} that every regular Bonnet-Myers sharp graph of diameter $3$ is self-centered. Thus, $G$ must be one of the four graphs if it is regular.
We will conclude this section by presenting infinitely many examples of irregular Bonnet-Myers sharp graphs of diameter $3$.

In this section, we assume $G$ is a Bonnet-Myers sharp graph of diameter $3$ and $(x,y)$ is a pair of poles.

\begin{lemma}\label{lem:neighbor relation}
For any $u \in N(x)$ and $v \in N(y)$, we have
\begin{align}
d^0_u &= 2d^+_u-4,\\
d^0_v &= 2d^-_v-4.
\end{align}
In particular, $d^0_u$ and $d^0_v$ are divisible by $2$.
\end{lemma}

\begin{proof}
By Lemma \ref{lem:xi}, we have 
\begin{equation}\label{eq:15}
d^+_u - d^-_u=\frac{1}{3}d_{u}.
\end{equation}
Observe that 
\begin{equation}\label{eq:16}
d_{u}=d^+_u+d^-_u+d^0_u.
\end{equation}
Since $d^-_u=1$, combining the Equations \ref{eq:15} and \ref{eq:16}, we have
\[d^0_u = 2d^+_u-4\]
The proof for the second item follows a symmetric argument.
\end{proof}

\begin{lemma} \label{lem:middle regular1}
 For any $u \in N(x)$ and $v \in N(y)$ such that $uv \in E(G)$, we have $d_u=d_v$.      
\end{lemma}

\begin{proof}
Suppose $d_u \neq d_v$. Without loss of generality, assume $d_u < d_v$. Let $\lcm(d_u,d_v)=c_ud_u=c_vd_v$. We have $c_u >c_v$ since $d_u < d_v$. Let $\sigma$ be an optimal coupling between $\mu_u$ and $\mu_v$. 

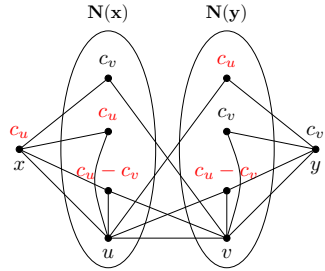
\begin{figure}[htb]
\begin{center}
     \resizebox{4.5cm}{!}{\begin{tikzpicture}[scale=1, Wvertex/.style={circle, draw=black, fill=white, scale=3}, bvertex/.style={circle, draw=black, fill=black, scale=0.3},rvertex/.style={circle, draw=red, fill=red, scale=0.2}]

{
\node[bvertex, label={[font=\small] below:$x$},  label={[font=\small] above:\textcolor{red}{$c_u$}}] (v1) at (-0.5,0) {};
\node[bvertex, label={[font=\small] below:$u$}] (v2) at (1,-1.5) {};
\node[bvertex, label={[font=\small] below:$v$}] (v3) at (3,-1.5) {};
\node[bvertex, label={[font=\small] below:$y$}, label={[font=\small] above:{$c_v$}}] (v4) at (4.5,0) {};
\node[bvertex, label={[font=\small] above:\textcolor{red}{$c_u-c_v$}}] (v5) at (3,-0.7) {};
\node[bvertex, label={[font=\small] above:\textcolor{red}{$c_u$}}] (v6) at (3,1.2) {};
\node[bvertex, label={[font=\small] above:{$c_v$}}] (v7) at (3,0.3) {};
\node[bvertex, label={[font=\small] above:\textcolor{red}{$c_u-c_v$}}] (v8) at (1,-0.7) {};
\node[bvertex, label={[font=\small] above:{$c_v$}}] (v9) at (1,1.2) {};
\node[bvertex, label={[font=\small] above:\textcolor{red}{$c_u$}}] (v10) at (1,0.3) {};
\draw(1,0) ellipse (0.8 and 2);
\draw (3,0) ellipse (0.8 and 2);
\draw (v2)--(v3);
\draw (v2)--(v5);
\draw (v2)--(v6);
\draw (v3)--(v5);
\draw (v3)--(v8);
\draw (v3)--(v9);
\draw (v2)--(v8);
\draw (v1)--(v8);
\draw (v1)--(v2);
\draw (v1)--(v9);
\draw (v1)--(v10);
\draw (v4)--(v3);
\draw (v4)--(v5);
\draw (v4)--(v6);
\draw (v4)--(v7);

\node[scale=0.8] (l1) at (1,2.3) {{$\bf N(x)$}};
\node[scale=0.8] (l1) at (3,2.3) {{$\bf N(y)$}};
\draw (v3) .. controls (3.3,-0.5) .. (v7);
\draw (v2) .. controls (0.7,-0.5) .. (v10);
}
\end{tikzpicture}}    
\end{center}
    \caption{ Mass distributions $\mu_u$ (in red color) and $\mu_v$ (in black color).}
    \label{fig:example2}
\end{figure}

For $z \in N^+(u)$, if $\sigma(z,w) \ne 0$, then by Lemma \ref{lem:duality2} we have $f(w)-f(z)=d(z,w)$. Since $f(z)=2$ and $d(z,w)\geq 1$, we have $f(w)=3$ so we must have $w=y$. This implies that the masses from $N^+(u)$ are all transported to $y$. We have 
\[c_u|(N(u) \setminus N[v]) \cap N(y)|+(c_u-c_v)(d^+_u-|(N(u) \setminus N[v]) \cap N(y)|)\leq c_v.\]
Since $c_u >c_v$, we must have $|(N(u) \setminus N[v]) \cap N(y)|=0$. Thus, we have 
\[(c_u-c_v)d^+_u\leq c_v.\]
It implies
\begin{equation}\label{eq:dvdu1}
   \left(\frac{d_v}{d_u}-1\right)d^+_u\leq 1. 
\end{equation}
By Lemma \ref{lem:xi}, we have 
\begin{align}
\frac{1}{3}d_u&=d^+_u-d^-_u=d^+_u-1,\\
\frac{1}{3}d_v&=d^-_v-d^+_v=d^-_v-1.
\end{align}
It implies
\begin{equation}\label{eq:dvdu2}\frac{d_v}{d_u}=\frac{d^-_v-1}{d^+_u-1}.
\end{equation}
Plugging Equation \eqref{eq:dvdu2} into Equation \eqref{eq:dvdu1}, we get
\begin{equation}\label{eq:dvdu3}
   \left(\frac{d^-_v-1}{d^+_u-1}-1\right)d^+_u\leq 1. 
\end{equation}
Solving $d^-_v$, we get
\[d^-_v\leq d^+_u + \frac{d^+_u-1}{d^+_u}< d^+_u+1.\]

On the other hand, since $d_v>d_u$, we have 
\[\frac{d^-_v-1}{d^+_u-1}>1.\]
This gives
\[d^-_v>d^+_u.\]
Since both $d^+_u$ and $d^-_v$ are integers, it is impossible to have $d^+_u<d^-_v<d^+_u+1$. Contradiction! Hence, we have $d_u=d_v$.

\end{proof}

\begin{lemma}\label{lem:middle regular2}
 All vertices in $V(G) \setminus \{x,y\}$ have the same degree.     
\end{lemma}

\begin{proof}
By Lemma \ref{lem:middle regular1}, the degree of vertices are the same for any connected component of $G\setminus\{x,y\}$.
It suffices to show $G\setminus \{x,y\}$ is connected.
For any vertex $u_1\in N(x)$, by Lemma \ref{lem:perfect matching}, there is a perfect matching in $H_1(x,u_1)$. This implies that
for any $u_2\in N(x)$ there exists a vertex $v\in N(u_1)\cap N(y)$ such that
$u_2v$ is an edge. Hence, $G\setminus\{x,y\}$ is connected.
\end{proof}

\begin{proof}[Proof of theorem \ref{thm:structural thm}]
Proof of Item 1: For $u \in N(x)$, we have $d^+_u-d^-_u=\frac{1}{3}d_u$ by Lemma \ref{lem:xi}. We observe that $d^-_u=1$ so $d_u=3d^+_u-3$. Similarly, we have $d_v=3d^-_v-3$ for $v \in N(y)$. By Lemma \ref{lem:middle regular2}, we have $d_u=d_v$ so $d^+_u=d^-_v$. Consider the edges between $N(x)$ and $N(y)$. On one hand, each vertex in $N(x)$ has $d^+_u$ neighbors in $N(y)$. On the other hand, each vertex in $N(y)$ has $d^-_v$ neighbors in $N(x)$. Thus, we have
\[d_xd^+_u=d_yd^-_v.\] 
Hence, we must have $d_x=d_y=\frac{n-2}{2}$. 

Now we will show that $G$ only has a unique pair of poles $x,y$. If $G$ has another pair of poles, say $x',y' \in V(G) \setminus \{x,y\}$, then we have that $d_{x'}=d_{y'}=\frac{n-2}{2}$ using the same argument above. Since all vertices in $ V(G) \setminus \{x,y\}$ have the same degree, we have $G$ is $\frac{n-2}{2}$-regular. Contradiction to $G$ being irregular.

Proof of Item 2: Let $v \in N(y)$. We observe that it is sufficient to show that $|(N(x) \cap N(v)) \setminus N[u]| \leq 1$ for any $u \in N(x) \cap N(v)$. Let $u \in N(x) \cap N(v)$. Then from Lemma \ref{lem:middle regular2}, we have $d_u=d_v$. By Lemma \ref{lem:private neighbor}, we have
\[|(N(x) \cap N(v)) \setminus N[u]| \leq \frac{d_v}{d_u}=1.\]

Proof of Item 3:  We can only consider a vertex $u \in N(x)$ because all vertices in $V(G) \setminus \{x,y\}$ have the same degree. From Lemma \ref{lem:neighbor relation}, we have $d^0_u=2d^+_u-4$ and $d^0_u$ is divisible by $2$. Thus, we have $d^0_u=2r$ and $d^+_u=r+2$ for some $r \in \mathbb{N}$. Then we have $d_u=d^+_u+d^0_u+d^-_u=r+2+2r+1=3(r+1)$. Let $t=|N(x) \setminus N(u)|$. Then $d_x=2r+t$.

Now we will show that $r$ and $t$ must satisfy $1\leq t \leq \frac{r}{2}+2$. Let $u \in N(x)$. Since $G$ is irregular, by Lemma \ref{lem:pole} we must have $d_x < d_u$. Let $\lcm(d_x,d_u)=c_xd_x=c_ud_u$. We have $c_x > c_u$. We consider the curvature of the edge $xu$. By Lemma \ref{lem:perfect matching}, there is a perfect matching in $H_1(x,u)$. Suppose $N(x) \setminus N[u] \ne \emptyset$. Each vertex in $N(x) \setminus N[u]$ transfers its masses to at least two vertices in $N(u) \setminus N[x]=N^+(u)$ because of $c_x>c_u$. This implies each vertex in $N(x) \setminus N[u]$ is adjacent to at least two vertices in $N^+(u)$. We observe that for any two different vertices $u'_1,u'_2 \in N(x) \setminus N[u]$, the two sets $N(u'_1) \cap N^+(u)$ and $N(u'_2) \cap N^+(u)$ must be disjoint. If not, then there exists a $v \in N^+(u)$ such that $|(N(x) \cap N(v)) \setminus N[u]| \geq 2$. This is impossible by item 2. Therefore, we have
\[2(t-1)=2|N(x) \setminus N[u]| \leq \sum_{u' \in N(x) \setminus N[u]} |N(u') \cap N^+(u)| \leq |N^+(u)|=r+2.\]
This gives us the upper bound of $t$. The case of $N(x) \setminus N[u]=\emptyset$ gives us the lower bound.
\end{proof}

\subsection{Construction of irregular diameter 3 Bonnet-Myers sharp graphs}
Theorem \ref{thm:structural thm} states that all irregular diameter 3 Bonnet-Myers sharp graphs have certain structures. In this section, we will show such graphs $G$ exist for $t=1,2$ and any $r\geq 1$. Let $x$ and $y$ be a pair of poles of $G$.  
We label the neighbors of $x$ by $\{u_0,u_1,.....,u_{2r+t-1}\}$, and the neighbors of $y$ by $\{v_0,v_1,.....,v_{2r+t-1}\}$.
Here we assume that indexes are in $\mathbb{Z}_{2r+t}$, i.e., $u_{2r+t+i}=u_i$ and $v_{2r+t+j}=v_j$.
We further assume $u_iv_j\in \E(G)$ if $i\leq j\leq i+r+1$. Let $G[N(x)]$ (and $G[N(y)]$) be a $2r$-regular graph as follows:

\begin{description}
    \item[$t=1$]: We set $G[N(x)]$ (and $G[N(y)]$) being the complete graph $K_{2r+1}$.
    \item[$t=2$]: We set $G[N(x)]$ (and $G[N(y)]$) being the complete graph $K_{2r+2}$ with a perfect matching $M_1$ (and $M_2$) removed, respectively.
    Here the matching $M_1$ in $N(x)$ is given by $u_iu_{i+r+1}$ for $i=0,1,2,\ldots, r$ while  the matching $M_2$ in $N(y)$ is given by $v_iv_{i+r+1}$ for $i=0,1,2,\ldots, r$. 
\end{description}

\begin{proposition}
    The graphs $G$ constructed above for $t=1,2$, and any $r\geq 1$ are diameter 3 Bonnet-Myers sharp graphs.
\end{proposition}
\begin{proof}
    The graph $G$ clearly has diameter $3$. It suffices to show $\kappa(x,y)\geq \frac{2}{3}$ for any edge $xy$ in $G$. (This implies $\kappa_{min}\geq \frac{2}{3}$ and $diam(G)\leq \frac{2}{\kappa_{min}}\leq 3$. Since $diam(G)=3$, we must have
$\kappa_{min}=\frac{2}{3}.$)

Graph $G$ has three type of edges.
\begin{enumerate}
    \item $xu_i$ and $yv_j$, for $i,j\in \mathbb{Z}_{2r+t}$.
    \item $u_iv_j$, for $i\leq j\leq i+r+1$.
    \item $u_{i_1}u_{i_2}$ and $v_{j_1}v_{j_2}$.
\end{enumerate}

For the first type of edges, by the symmetry of $G$, it suffices to show
$\kappa(x,u_0)\geq \frac{2}{3}$.
We have $d_x=2r+t$ and $d_{u_0}=3r+3$. Since $d_{u_0}<2d_x$, we have
$c_x<2c_{u_0}$ and 
\begin{align*}
\mu_x(u)&= \begin{cases}
    c_x & \mbox{ if } t=2 \mbox{ and } u=u_{r+1};\\
    c_x-c_{u_0}& u\in N(x) \mbox{ but not the first case};\\ 
    0& \mbox{ otherwise.}
\end{cases}\\
\mu_{u_0}(v)&= \begin{cases}
      c_{u_0} &\mbox{ if $v=v_j$ for } 0\leq j \leq r+1;\\
      0 & \mbox{ otherwise.}
\end{cases}
\end{align*}

We use the following greedy algorithm to assign the coupling $\sigma$.
Assume $u_i$ has $\mu_x(u_i)$ units of masses at the beginning of the algorithm while for each
$0\leq j\leq r+1$, $v_j$ has an empty bin of capacity $c_{u_0}$.

\begin{table}[ht]
\centering
\begin{minipage}{0.7\textwidth}
    \begin{tabbing}
    MMMM\=MMMM\=MMMM\=\kill
    {\bf Algorithm:}\\
    \+\>$i \leftarrow -r$\\
    $j \leftarrow 0$ \\
    while (not complete)\{\\
    \>Transfer as many masses from $u_i$ to $v_j$ as possible\\
    \>if $u_i$ is empty, $i\leftarrow i+1$\\
    \>if $v_j$ is full, $j\leftarrow j+1$\\
    \}
    \end{tabbing}
   \end{minipage}
\caption{A greedy algorithm to transfer masses from $\mu_x$ to $\mu_{u_0}$.}
\end{table}

We claim that during the process, we always have $0\leq j-i\leq r+1$.

The claim is true at the beginning. Since $c_x-c_{u_0}\leq c_{u_0}$, $i$ increases faster that $j$.  
Thus $j-i$ decreases during the process. When the algorithm stops, we have $i=r+t-1$ and $j=r+1$. Thus $j-i=2-t\geq 0$. The proof of the claim is finished. Therefore, $u_iv_j$ is always an edge of $G$. 

Therefore, we show there is a coupling $\sigma$ with cost $(r+2)c_{u_0}$.
By Theorem \ref{thm:localstructure}, we have

\begin{align*}
    \kappa(x,u_i)&\geq 1+\frac{1}{d_{u_0}}-\frac{C(\sigma)}{\lcm(d_x, d_{u_0})}.\\
    &= 1+\frac{1}{d_{u_0}}-\frac{c_{u_0}(r+2)}{\lcm(d_x, d_{u_0})}\\
    &= 1+\frac{1}{3(r+1)}-\frac{r+2}{3(r+1)}\\
     &= \frac{2}{3}.
\end{align*}

Now consider second kind of edges. By symmetry, it suffices to show $\kappa(u_0,v_j)\geq \frac{2}{3}$ for $0\leq j \leq r+1$.
Since $d_{u_0}=d_{v_j}=3r+3$, we have $c_{u_0}=c_{v_j}=1$. Note that for any $u \in N(u_0)\setminus N[v_j]$, have $\mu_{u_0}(u)=1$. For any $v \in N(v_j)\setminus N[u_0]$, we have $\mu_{v_j}(v)=1$. 

For the case of $t=1$, we have
\begin{align*}
    X(u_0,v_j)&=N(u_0)\setminus N[v_j]=\{x\}\cup \{u_{j+1},u_{j+2},\ldots, u_{j+r-1} \},\\
    Y(u_0,v_j)&=N(v_j)\setminus N[u_0]=\{y\}\cup \{v_{r+2},v_{r+3},\ldots, v_{2r} \}.
\end{align*}
Transfer $1$ unit of mass from $u_{j+k}$ to 
$v_{r+1+k}$ for $k=1,2,\ldots,r-1$. Since \[0\leq (r+1+k)-(j+k)\leq r+1,\] each move has cost $1$. Finally, transfer $1$ unit of mass from $x$ to $y$ with cost 3.

The total cost is \[C(\sigma)=(r-1)+3=r+2.\]

For the case of $t=2$, we have the following three cases: 
\begin{enumerate}
\item For $1 \leq j \leq r$, we have
\begin{align*}
    X(u_0,v_j)&=N(u_0)\setminus N[v_j]=\{x\}\cup \{u_{j+1},u_{j+2},\ldots, u_{j+r} \} \setminus \{u_{r+1}\},\\
    Y(u_0,v_j)&=N(v_j)\setminus N[u_0]=\{y\}\cup \{v_{r+2},v_{r+3},\ldots, v_{2r+1} \} \setminus \{v_{j+r+1}\}.
\end{align*}

Transfer $1$ unit of mass from the vertex in set $X(u_0,v_j)\setminus \{x\}$ with the lowest non-empty index to the vertex in set $Y(u_0,v_j)\setminus \{y\}$ with the lowest non-full index. Continue this process until all vertices in $X(u_0,v_j)\setminus \{x\}$ are empty and all vertices in $Y(u_0,v_j)\setminus \{y\}$ are full. Observe that the difference of indexes during this process is 
$r+1-j+\varepsilon$ where $\varepsilon \in \{-1,0,1\}$ depending
on the relative positions of the exclusive
vertices $u_{r+1}$ and $v_{j+r+1}$ in the sequence. Since
\[ 0\leq r+1-j+\varepsilon\leq r+1,\]
each move has cost 1. By the end, a total of $r-1$ units of masses will have been moved from $X(u_0,v_j)\setminus \{x\}$ to $Y(u_0,v_j)\setminus \{y\}$. Finally, transfer $1$ unit of mass from $x$ to $y$ with cost $3$.
The total cost is \[C(\sigma)=r-1+3=r+2.\]

\item
For $j=0$, we have
\begin{align*}
    X(u_0,v_0)&=N(u_0)\setminus N[v_0]=\{x\}\cup \{u_{1},u_{2},\ldots, u_{r} \} \cup \{v_{r+1}\},\\
    Y(u_0,v_0)&=N(v_0)\setminus N[u_0]=\{y\}\cup \{v_{r+2},v_{r+3},\ldots, v_{2r+1}\} \cup \{u_{r+1}\}.
\end{align*}
Transfer $1$ unit of mass from 
$u_{k}$ to $v_{r+1+k}$ for $k=1,2,\ldots, r$. Transfer $1$ unit of mass from $x$ to $u_{r+1}$. Transfer $1$ unit of mass from $v_{r+1}$ to $y$.
The total cost is
\[C(\sigma)=r+1+1=r+2.\]

\item For $j=r+1$, we have
\begin{align*}
    X(u_0,v_{r+1})&=N(u_0)\setminus N[v_{r+1}]=\{x\}\cup \{u_{r+2},u_{r+3},\ldots, u_{2r+1} \} \cup \{v_0\},\\
    Y(u_0,v_{r+1})&=N(v_{r+1})\setminus N[u_0]=\{y\}\cup \{v_{r+2},v_{r+3},\ldots, v_{2r+1}\} \cup \{u_{r+1}\}.
\end{align*}

Transfer $1$ unit of mass from 
$u_{r+1+k}$ to $v_{r+1+k}$ for $k=1,2,\ldots, r$. Transfer $1$ unit of mass from $x$ to $u_{r+1}$. Transfer $1$ unit of mass from $v_0$ to $y$.
The total cost is 
\[C(\sigma)=r+1+1=r+2.\]
\end{enumerate}

By Theorem \ref{thm:localstructure}, we have
\begin{align*}
\kappa(u_0,v_j)&\geq 1+\frac{1}{d_{v_j}}-\frac{C(\sigma)}{\lcm(d_{u_0}, d_{v_j})}\\
&=1+\frac{1}{3(r+1)}-\frac{r+2}{3(r+1)}\\
&=\frac{2}{3}.
\end{align*}  

Now consider third kind of edges. By symmetry, it suffices to show $\kappa(u_0,u_i)\geq \frac{2}{3}$.
We can further assume $i\leq \frac{2r+t}{2}-1$. (If $i>\frac{2r+t}{2}-1$, we can rotate the pair of vertices by
$2r+t-i$ so that $u_0u_i$ is a shorter arc on the $(2r+t)$-cycle.)

Since $d_{u_0}=d_{u_i}=3r+3$, we have $c_{u_0}=c_{u_i}=1$. Note that for any $u \in N(u_0)\setminus N[u_i]$, have $\mu_{u_0}(u)=1$. For any $v \in N(u_i)\setminus N[u_0]$, we have $\mu_{u_i}(v)=1$. 

For the case of $t=1$, we have
\begin{align*}
    X(u_0,u_i)&=N(u_0)\setminus N[u_i]=\{v_0,v_1,\ldots, v_{i-1} \},\\
    Y(u_0,u_i)&=N(u_i)\setminus N[u_0]=\{v_{r+2},v_{r+3},\ldots, v_{i+r+1} \}.
\end{align*}
For $j=0,1,\ldots, i-1$, we transfer
$1$ unit of mass for $v_j$ to $v_{r+2+j}$. Since $G[N(y)]$ is a complete graph, each move has cost $1$. Since $i\leq \frac{2r+1}{2}-1=r-\frac{1}{2}$ and $i$ is an integer,
we have $i\leq r-1$. We have
\[C(\sigma)\leq i\leq r-1<r+2.\]

For the case of $t=2$, we have
\begin{align*}
    X(u_0,u_i)&=N(u_0)\setminus N[u_i]=\{u_{i+r+1}, v_0,v_1,\ldots, v_{i-1} \},\\
    Y(u_0,u_i)&=N(u_i)\setminus N[u_0]=\{u_{r+1}, v_{r+2},v_{r+3},\ldots, v_{i+r+1} \}.
\end{align*}
For $j=0,1,\ldots, i-1$, we transfer
$1$ unit of mass from $v_j$ to $v_{r+2+j}$.
Since $r+2+j-j=r+2$,
$v_jv_{r+2+j}$ is an edge.
Each move has cost $1$.

Finally, we transfer $1$ unit of mass
from $u_{i+r+1}$ to $u_{r+1}$. Since $i+r+1-(r+1)=i$ and $i\leq \frac{2r+2}{2}-1=r$, $u_{i+r+1}u_{r+1}$ is an edge.
The cost of this move is $1$.
We have
\[C(\sigma)\leq i+1\leq r+1<r+2.\]

In either case, we have
\begin{align*}
\kappa(u_0,u_i)&\geq 1+\frac{1}{d_{u_0}}-\frac{C(\sigma)}{\lcm(d_{u_0}, d_{u_i})}\\
&> 1+\frac{1}{3(r+1)}-\frac{r+2}{3(r+1)}\\
&=\frac{2}{3}.
\end{align*}  
The proof of this proposition is finished.
\end{proof}

\section{$C_3$-free Bonnet-Myers sharp graphs}
In this section, we assume $G$ is a $C_3$-free Bonnet-Myers sharp graph with $diam(G)=L$, and $x$ is a pole of $G$.

\begin{lemma} \label{lem:C3free}
We have
\begin{enumerate}
    \item For any vertex $u\in V(G)$, we have $d_u \leq L$. If $d_u=L$, then for any $v\in N^+(u)\cup N^-(u)$, there is a perfect matching in $H_1(v,u)$.
    \item For any $u\in N(x)$, $d_u=L$.
    \item For any $u\in N_i(x)$, we have $d^-_u\leq i$. Moreover, if $d^-_u=i$, then 
    $d^0_u=0$ and $d^+_u=L-i$.
\end{enumerate}
\end{lemma}
\begin{proof}
We first prove item 1.
Since $G$ is triangle-free, we have $|N(u)\cap N(v)|=0$ for any edge $uv$.
Applying Lemma \ref{lem:d upper bound}, we have
\[d_u \leq \left(\frac{|N(u) \cap N(v)|+2}{2}\right)L =L.\]
The equality holds if there is a perfect matching in $H_1(v,u)$.
 
Now we prove item 2.
For $u \in N(x)$, we have $d^-_u=1$. Since $G$ is triangle-free, we must have $d^0_u=0$. Applying Lemma \ref{lem:biratio},
we get $d^+_u=L-1$.
Thus $d_u=L$. 

Finally, we prove item 3.
By Lemma \ref{lem:xi}, we have
\[d^+_u-d^-_u=\left(1-\frac{2i}{L}\right)d_{u}.\]
Since $d_u=d^+_u+ d^0_u+ d^-_u$, we have
\[2d^-_u + d^0_u= d_u - (d^+_u-d^-_u) = \frac{2i}{L}d_{u} \leq 2i.\]
Therefore, we have $d^-_u \leq i$. If $d^-_u = i$, then $d^0_u=0$. By Lemma \ref{lem:biratio}, we have $d^+_u=L-i$ and $d_{u}=L$.
\end{proof}

\begin{lemma}\label{lem:N2bijective}
We have
\begin{enumerate}
    \item If $d_x >\frac{1}{2}L$, then for any $u \in N_2(x)$, we have $d^-_u=2$.
    \item If $d_x >\frac{2}{3}L$, then for any two vertices $u_1, u_2 \in N_2(x)$,
    we have $N(u_1)\cap N(x) \not =N(u_2)\cap N(x)$.
\end{enumerate}

\end{lemma}

\begin{proof}
For item 1, let $u \in N_2(x)$, we have $d^-_u \leq 2$ from Lemma \ref{lem:C3free}. It is sufficient to show $d^-_u \geq 2$. We will prove it by contradiction. Otherwise we have $d^-_u=1$. Let $v$ be the only vertex in $N(u)\cap N(x)$. Consider the curvature of the edge $xv$.
Let $\lcm(d_x,d_v)=c_xd_x=c_vd_v$. Since $d_x \in \left(\frac{1}{2}L, L\right]$ and $d_v=L$, we have $c_{v} \leq c_x < 2c_{v}$. By Lemma \ref{lem:perfect matching}, there exists a perfect matching in $H_1(x,v)$, i.e., all non-zero entries of the optimal coupling $\sigma$ are equal to 1. However, the masses from $v$ are not sufficient to fill up
the vertex $u$ because of $c_x-c_v<c_v$ (see Figure \ref{fig:lemma18} (i)). Contradiction!

\begin{figure}[h]
   \begin{center}
        \begin{minipage}{.35\textwidth}
        		\resizebox{4.5cm}{!}{\begin{tikzpicture}[scale=1, Wvertex/.style={circle, draw=black, fill=white, scale=3}, bvertex/.style={circle, draw=black, fill=black, scale=0.3},rvertex/.style={circle, draw=red, fill=red, scale=0.2}]

{
\node[bvertex, label={[font=\small] below:$x$}] (v1) at (-0.5,0) {};
\node[bvertex, label={[font=\small] below:$v$}, label={[font=\small] above:\textcolor{red}{$c_x-c_{v}$}}] (v2) at (1,-1) {};
\node[bvertex, label={[font=\small] below:$u$}, label={[font=\small] above:{$c_{v}$}}] (v4) at (3,-1) {};
\node [label={[font=\small] above:$......$}] (uk'') at (4.5,-0.3) {};

\draw(1,0) ellipse (0.8 and 2);
\draw (3,0) ellipse (0.8 and 2);
\draw (v1)--(v2);
\draw (v2)--(v4);

\node[scale=0.8] (l1) at (1,2.3) {{$\bf N(x)$}};
\node[scale=0.8] (l1) at (3,2.3) {{$\bf N_2(x)$}};

}
\end{tikzpicture}}
        \end{minipage}
            \hfil
        \begin{minipage}{.35\textwidth}
        		\resizebox{4.5cm}{!}{\begin{tikzpicture}[scale=1, Wvertex/.style={circle, draw=black, fill=white, scale=3}, bvertex/.style={circle, draw=black, fill=black, scale=0.3},rvertex/.style={circle, draw=red, fill=red, scale=0.2}]

{
\node[bvertex, label={[font=\small] below:$x$}] (v1) at (-0.5,0) {};
\node[bvertex, label={[font=\small] below:$v_1$}, label={[font=\small] above:\textcolor{red}{$c_x-c_{v_1}$}}] (v2) at (1,-1) {};
\node[bvertex, label={[font=\small] below:$v_2$}, label={[font=\small] above:\textcolor{red}{$c_x$}}] (v3) at (1,1) {};
\node[bvertex, label={[font=\small] below:$u_1$}, label={[font=\small] above:{$c_{v_1}$}}] (v4) at (3,-1) {};
\node[bvertex, label={[font=\small] below:$u_2$}, label={[font=\small] above:{$c_{v_1}$}}] (v5) at (3,1) {};
\node [label={[font=\small] above:$......$}] (uk'') at (4.5,-0.3) {};

\draw(1,0) ellipse (0.8 and 2);
\draw (3,0) ellipse (0.8 and 2);
\draw (v1)--(v2);
\draw (v1)--(v3);
\draw (v2)--(v4);
\draw (v2)--(v5);
\draw (v3)--(v4);
\draw (v3)--(v5);

\node[scale=0.8] (l1) at (1,2.3) {{$\bf N(x)$}};
\node[scale=0.8] (l1) at (3,2.3) {{$\bf N_2(x)$}};

}
\end{tikzpicture}}
        \end{minipage}
    \end{center}
    \caption{Two impossible cases: (i) $N(u)\cap N(x)=\{v\}$.
   (ii) $N(u_1)\cap N(x) =N(u_2)\cap N(x)=\{v_1,v_2\}$.}
    \label{fig:lemma18}
\end{figure}
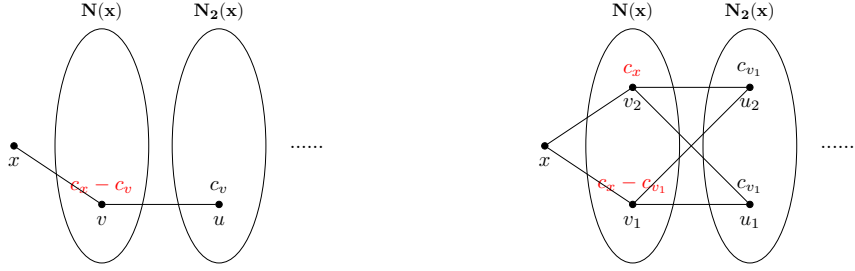

Now we prove item 2 by contradiction.  Suppose that there exist two vertices 
$u_1, u_2 \in N_2(x)$, such that $N(u_1)\cap N(x) =N(u_2)\cap N(x)=\{v_1,v_2\}$ (see Figure \ref{fig:lemma18} (ii)).

Consider the curvature of the edge $xv_1$.
Let $\lcm(d_x,d_{v_1})=c_xd_x=c_{v_1}d_{v_1}$. Since $d_x \in \left(\frac{2}{3}L,L\right]$ and $d_{v_1}=L$, we have $c_{v_1} \leq c_x < \frac{3}{2}c_{v_1}$. By Lemma \ref{lem:perfect matching}, there is a perfect matching in $H_1(x,v_1)$. The matching edges are from left to right (see Corollary \ref{cor:matchingdirection}).
Since the masses from $v_1, v_2$ need to be sufficient to fill up $u_1,u_2$, we have
\[c_x-c_{v_1}+c_x= |S_{v_1}\cup S_{v_2}| \geq |S_{u_1}\cup S_{u_2}| 
=2c_{v_1}.\]
It implies $c_x\geq \frac{3}{2}c_{v_1}$. Contradiction!
\end{proof}

\begin{proof}[Proof of Theorem \ref{thm:hypercube}]
We first show $d_x=L$. By Lemma \ref{lem:N2bijective}, all vertices in $N_2(x)$
have exactly two neighbors in $N(x)$, and these pairs of neighbors are distinct.
Thus, we have
\begin{equation} \label{eq:n2up}
   |N_2(x)|\leq \binom{d_x}{2}. 
\end{equation}

Now consider the number of edges between $N(x)$ and $N_2(x)$. On one hand,
each vertex in $N(x)$ has $L-1$ neighbors in $N_2(x)$. On the other hand, each
vertex in $N_2(x)$ has $2$ neighbors in $N(x)$. We have
\begin{equation}\label{eq:n2}
    d_x(L-1)=2|N_2(x)|.
\end{equation}

Combining Equations \eqref{eq:n2up} and \eqref{eq:n2}, we get
  \[ d_x(L-1) \leq 2 \binom{d_x}{2}. \]
  It implies $L\leq d_x$. By Lemma \ref{lem:C3free}, $d_x\leq L$. Therefore, we must have
  $d_x=L$.

We will establish an isomorphism $g$ between $G$ and $Q_L$. Label the vertices of $Q_L$ by the subsets of $[L]$. First, we define $g(x)=\emptyset$. Denote the set of neighbors of $x$ by $u_1,u_2,\dots,u_L$ and define $g(u_k)=\{k\}$ for $1 \leq k \leq L$. 

By Lemma \ref{lem:N2bijective}, every vertex in $N_2(x)$ has a distinct pair of neighbors in $N(x)$. For $v \in N_2(x)$, we define $g(v)=g(u_1) \cup g(u_2)$ whenever $N(v) \cap N(x)=\{u_1,u_2\}$. By Equation \ref{eq:n2}, we have $|N_2(x)|=\binom{L}{2}$ so every distinct pair of vertices in $N(x)$ have a common neighbor in $N_2(x)$. Thus, $g$ is an isomorphism between $G\left[\cup_{k=0}^{2} N_k(x)\right]$ and $Q_L\left[\cup_{k=0}^{2} \binom{[L]}{k}\right]$.

Now we will continue to define the map $g$ from $N_k(x)$ to $\binom{[L]}{k}$ for $3 \leq k \leq L$ by induction on $k$. Suppose that the isomorphism $g$ from $G\left[\cup_{k=0}^{i-1} N_k(x)\right]$ to $Q_L\left[\cup_{k=0}^{i-1} \binom{[L]}{k}\right]$ has already defined. In particular, the induced bipartite graph on $(N_{i-2}(x), N_{i-1}(x))$ is $C_4$-free.
(Such a $C_4$ is called a {\em butterfly} graph. There is no butterfly graph between two consecutive levels in the hypercube.) Now we need to define the map $g$ from $N_i(x)$ to $\binom{[L]}{i}$. 

First, we will show that $d^-_v=i$ for any $v \in N_i(x)$.  By Lemma \ref{lem:C3free}, it suffices to show $d^-_v\geq i$. 
Let $w \in N^-(v) \subseteq N_{i-1}(x)$. Take any $z \in N^-(w) \subseteq N_{i-2}(x)$. Since the isomorphism $g$ from $G[\cup_{k=0}^{i-1} N_k(x)]$ to $Q_L[\cup_{k=0}^{i-1} \binom{[L]}{k}]$ has already defined, we must have $d^-_w=i-1$ and $d^-_z=i-2$. Then by Lemma \ref{lem:C3free}, we have $d_w=d_z=L$. Consider the curvature of the edge $zw$. Let $\lcm(d_z,d_w)=c_zd_z=c_wd_w$. We have $c_z=c_w=1$. Since $d_w=d_z=L$, we know there is a perfect matching in $H_1(z,w)$ by Lemma \ref{lem:C3free}. The matching edges are from left to right (see Corollary \ref{cor:matchingdirection}).
In particular, a vertex $w':=w'(z,w,v) \in N^+(z) \subseteq N_{i-1}(x)$ is matched to $v$ where $w' \ne w$, for each $z \in N^-(w)$.

Now we will show that different $z$ give different choices of $w'$. Suppose $z_1,z_2$ where $z_1 \ne z_2$ give the same choice of $w'$. Then we get a butterfly graph $z_1,z_2, w',w$
(see Figure \ref{fig:theorem7} (i)).
However, $Q_L$ does not contain the butterfly subgraph. Contradiction!
Since there are $i-1$ choices of $z$, we get $i-1$ choices of $w'$. These $i-1$ vertices together with $w$ are the neighbors of $v$ in $N_{i-1}(x)$. Therefore, $d^-_v\geq i$. 
Thus, $d^-_v=i$. This implies $d_v=L$ by Lemma \ref{lem:C3free}.

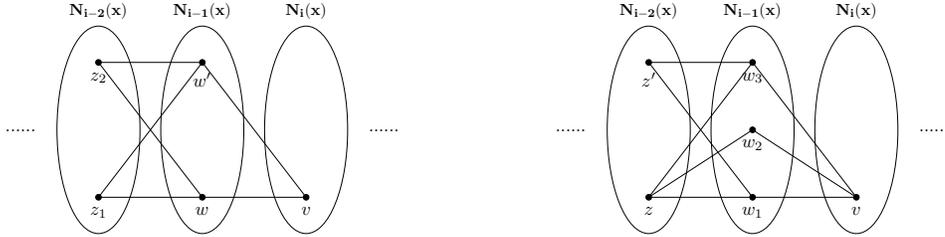
\begin{figure}[h!]
 \begin{center}
        \begin{minipage}{.40\textwidth}
        		\resizebox{5.5cm}{!}{\begin{tikzpicture}[scale=1, Wvertex/.style={circle, draw=black, fill=white, scale=0.3}, bvertex/.style={circle, draw=black, fill=black, scale=0.3},rvertex/.style={circle, draw=red, fill=red, scale=0.2}]

{
\node[bvertex, label={[font=\small] below:$z_1$}, label={[font=\small]}] (v1) at (-1,-1.3) {};
\node[bvertex, label={[font=\small] below:$z_2$}, label={[font=\small] }] (v2) at (-1,1.3) {};
\node[bvertex, label={[font=\small] below:$w$}] (v3) at (1,-1.3) {};
\node[bvertex, label={[font=\small] below:$w'$}] (v5) at (1,1.3) {};
\node[bvertex, label={[font=\small] below:$v$}] (v6) at (3,-1.3) {};
\node [label={[font=\small] above:$......$}] (uk'') at (4.5,-0.3) {};
\node [label={[font=\small] above:$......$}] (uk'') at (-2.5,-0.3) {};

\draw(-1,0) ellipse (0.8 and 2);
\draw(1,0) ellipse (0.8 and 2);
\draw (3,0) ellipse (0.8 and 2);
\draw (v1)--(v3);
\draw (v1)--(v5);
\draw (v2)--(v3);
\draw (v2)--(v5);
\draw (v6)--(v3);
\draw (v6)--(v5);

\node[scale=0.8] (l1) at (-1,2.3) {{$\bf N_{i-2}(x)$}};
\node[scale=0.8] (l1) at (1,2.3) {{$\bf N_{i-1}(x)$}};
\node[scale=0.8] (l1) at (3,2.3) {{$\bf N_i(x)$}};

}
\end{tikzpicture}}
        \end{minipage}
            \hfil
        \begin{minipage}{.40\textwidth}
        		\resizebox{5.5cm}{!}{\begin{tikzpicture}[scale=1, Wvertex/.style={circle, draw=black, fill=white, scale=3}, bvertex/.style={circle, draw=black, fill=black, scale=0.3},rvertex/.style={circle, draw=red, fill=red, scale=0.2}]

{
\node[bvertex, label={[font=\small] below:$z$}] (v1) at (-1,-1.3) {};
\node[bvertex, label={[font=\small] below:$z'$}] (v2) at (-1,1.3) {};
\node[bvertex, label={[font=\small] below:$w_1$}] (v3) at (1,-1.3) {};
\node[bvertex, label={[font=\small] below:$w_2$}] (v4) at (1,0) {};
\node[bvertex, label={[font=\small] below:$w_3$}] (v5) at (1,1.3) {};
\node[bvertex, label={[font=\small] below:$v$}] (v6) at (3,-1.3) {};
\node [label={[font=\small] above:$......$}] (uk'') at (4.5,-0.3) {};
\node [label={[font=\small] above:$......$}] (uk'') at (-2.5,-0.3) {};

\draw(-1,0) ellipse (0.8 and 2);
\draw(1,0) ellipse (0.8 and 2);
\draw (3,0) ellipse (0.8 and 2);
\draw (v1)--(v3);
\draw (v1)--(v4);
\draw (v1)--(v5);
\draw (v2)--(v3);
\draw (v2)--(v5);
\draw (v6)--(v3);
\draw (v6)--(v4);
\draw (v6)--(v5);

\node[scale=0.8] (l1) at (-1,2.3) {{$\bf N_{i-2}(x)$}};
\node[scale=0.8] (l1) at (1,2.3) {{$\bf N_{i-1}(x)$}};
\node[scale=0.8] (l1) at (3,2.3) {{$\bf N_i(x)$}};

}
\end{tikzpicture}}
        \end{minipage}
    \end{center}
\caption{Two impossible cases: (i) $z_1,z_2, w', w$ forms a butterfly subgraph. (ii) $z,z', w_1, w_3$ forms a butterfly graph.
}
    \label{fig:theorem7}
\end{figure}

For $v \in N_i(x)$, we define $g(v)=\cup_{k=1}^{i} g(w_k)$ whenever $N^-(v)=\{w_k: 1 \leq k \leq i\}$. We claim
that $|g(v)|=i$. We need to show that any two distinct vertices in $N^-(v) \subseteq N_{i-1}(x)$ have a common neighbor in $N_{i-2}(x)$ and no three distinct vertices in $N^-(v) \subseteq N_{i-1}(x)$ have a common neighbor in $N_{i-2}(x)$. Let $w_1,w_2 \in N^-(v) \subseteq N_{i-1}(x)$ where $w_1 \ne w_2$. Consider the curvature of the edge $w_1v$. Since $d_{w_1}=d_v=L$, there is a perfect matching in $H_1(w_1,v)$ by Lemma \ref{lem:C3free}. The matching edges are from left to right (see Corollary \ref{cor:matchingdirection}).
In particular, a vertex denoted by $z:=z(w_1,v,w_2)$, in $N^-(w_1) \subseteq N_{i-2}(x)$, is matched to $w_2$. Thus, $z$ is a common neighbor of $w_1$ and $w_2$ in $N_{i-2}(x)$. Suppose $z$ is a common neighbor of $w_1, w_2, w_3 \in N^-(v) \subseteq N_{i-1}(x)$ where $w_1,w_2,w_3$ are distinct (see Figure \ref{fig:theorem7} (ii)).
We observe that $z$ can only be matched to one of $w_2$ and $w_3$ in the perfect matching in $H_1(w_1,v)$ because of $c_{w_1}=c_v=1$. If $z$ is matched to $w_2$, then another vertex, denoted by $z'=:z'(w_1,v,w_3)$, in $ N^-(w_1) \subseteq N_{i-2}(x)$, must be matched to $w_3$. Then we get a butterfly graph $z,z',w_1,w_3$
(see Figure \ref{fig:theorem7} (ii)).
However, $Q_L$ does not contain the butterfly subgraph. Contradiction!
Hence, any two distinct vertices in $N^-(v) \subseteq N_{i-1}(x)$ have a common neighbor in $N_{i-2}(x)$ and no three distinct vertices in $N^-(v) \subseteq N_{i-1}(x)$ have a common neighbor in $N_{i-2}(x)$. Then this property implies $|g(v)|=i$. Hence, $g$ is an isomorphism from $G\left[\cup_{k=0}^{i} N_k(x)\right]$ to $Q_L\left[\cup_{k=0}^{i} \binom{[L]}{k}\right]$. The inductive proof is finished.

\end{proof}

Use the same technique, we  have completely classified $C_3$-free Bonnet-Myers sharp graphs with diameter $L=2,3,4,5$ (see table \ref{table1}). We omit the proof here. The second graph in the row where $L=4$ and $d_x=2$ is the first incidence featuring exactly two pairs of poles.
All of these graphs are bipartite. 
Additionally, we note that the maximum degree between any two adjacent vertices always matches the diameter and we conjecture this pattern holds for all $C_3$-free Bonnet-Myers sharp graphs. Based on these observations, we propose the following conjecture to conclude this paper.

\begin{conjecture}
Suppose $G$ is an $C_3$-free and Bonnet-Myers sharp graph with diameter $L$.
Then,
\begin{enumerate}
    \item $G$ is a bipartite graph.
    \item For any edge $uv$, we have $\max\{d_u,d_v\}=L$.
\end{enumerate}
\end{conjecture}
\begin{table}[htb]
    \centering
    \begin{tabular}{|c|c|c|}
\hline
 & $d_x=1$ & $P_3$\\
\cline{2-3}
$L=2$ & $d_x=2$ & $Q_2$\\
    \hline
    & $d_x=1$ & none \\
     \cline{2-3}
       $L=3$ & $d_x=2$ &  \tikz[scale=1, vertex/.style={scale=0.5, circle, draw=black, fill=black},
wvertex/.style={scale=0.5, circle, draw=black, fill=white}]{
\node[wvertex] (v1) at (0,0) {};
\node[wvertex](v2) at (1,0.5) {};
\node[wvertex] (v3) at (1,-0.5) {};
\node[wvertex] (v4) at (2,0.5) {};
\node[wvertex] (v5) at (2,-0.5) {};
\node[wvertex] (v6) at (3,0) {};
\draw (v1)--(v2);
\draw (v1)--(v3);
\draw (v2)--(v5);
\draw (v3)--(v4);
\draw (v2)--(v4);
\draw (v3)--(v5);
\draw (v4)--(v6);
\draw (v5)--(v6);
}\\
        \cline{2-3}
        &$d_x=3$ &  $Q_3$ \\
      \hline
    & $d_x=1$ & \tikz[scale=1, vertex/.style={scale=0.5, circle, draw=black, fill=black},
wvertex/.style={scale=0.5, circle, draw=black, fill=white}]{
\node[wvertex] (v1) at (0,0) {};
\node[wvertex](v2) at (1,0) {};
\node[wvertex] (v3) at (2,-0.5) {};
\node[wvertex] (v4) at (2,0) {};
\node[wvertex] (v5) at (2,0.5) {};
\node[wvertex] (v6) at (3,0) {};
\node[wvertex] (v7) at (4,0) {};
\draw (v1)--(v2);
\draw (v2)--(v3);
\draw (v2)--(v4);
\draw (v2)--(v5);
\draw (v3)--(v6);
\draw (v4)--(v6);
\draw (v5)--(v6);
\draw (v6)--(v7);
} \\
     \cline{2-3}
     $L=4$ & $d_x=2$ & \tikz[scale=1, vertex/.style={scale=0.5, circle, draw=black, fill=black},
wvertex/.style={scale=0.5, circle, draw=black, fill=white}]{
\node[wvertex] (v1) at (0,0) {};
\node[wvertex](v2) at (1,-0.5) {};
\node[wvertex] (v3) at (1,0.5) {};
\node[wvertex] (v4) at (2,-0.8) {};
\node[wvertex] (v5) at (2,0) {};
\node[wvertex] (v6) at (2,0.8) {};
\node[wvertex] (v7) at (3,-0.5) {};
\node[wvertex] (v8) at (3,0.5) {};
\node[wvertex] (v9) at (4,0) {};
\draw (v1)--(v2);
\draw (v1)--(v3);
\draw (v2)--(v4);
\draw (v2)--(v5);
\draw (v2)--(v6);
\draw (v3)--(v4);
\draw (v3)--(v5);
\draw (v3)--(v6);
\draw (v4)--(v7);
\draw (v4)--(v8);
\draw (v5)--(v7);
\draw (v5)--(v8);
\draw (v6)--(v7);
\draw (v6)--(v8);
\draw (v7)--(v9);
\draw (v8)--(v9);
} \tikz[scale=1, vertex/.style={scale=0.5, circle, draw=black, fill=black},
wvertex/.style={scale=0.5, circle, draw=black, fill=white}]{
\node[wvertex] (v1) at (0,0) {};
\node[wvertex](v2) at (1,-0.5) {};
\node[wvertex] (v3) at (1,0.5) {};
\node[wvertex] (v4) at (2,-0.8) {};
\node[wvertex] (v5) at (2,-0.3) {};
\node[wvertex] (v6) at (2,0.3) {};
\node[wvertex] (v7) at (2,0.8) {};
\node[wvertex] (v8) at (3,-0.5) {};
\node[wvertex] (v9) at (3,0.5) {};
\node[wvertex] (v10) at (4,0) {};
\draw (v1)--(v2);
\draw (v1)--(v3);
\draw (v2)--(v4);
\draw (v2)--(v5);
\draw (v2)--(v6);
\draw (v3)--(v5);
\draw (v3)--(v6);
\draw (v3)--(v7);
\draw (v4)--(v8);
\draw (v5)--(v8);
\draw (v6)--(v8);
\draw (v5)--(v9);
\draw (v6)--(v9);
\draw (v7)--(v9);
\draw (v8)--(v10);
\draw (v9)--(v10);
}\\
     \cline{2-3}  
      & $d_x=3$ & none \\
     \cline{2-3}
      &$d_x=4$ &  $Q_4$ \\
      \hline
      & $d_x=1$ & none \\
     \cline{2-3}
 $L=5$    & $d_x=2$ & \tikz[scale=1, vertex/.style={scale=0.5, circle, draw=black, fill=black},
wvertex/.style={scale=0.5, circle, draw=black, fill=white}]{
\node[wvertex] (v1) at (0,0) {};
\node[wvertex](v2) at (1,-0.5) {};
\node[wvertex] (v3) at (1,0.5) {};
\node[wvertex] (v4) at (2,-0.8) {};
\node[wvertex] (v5) at (2,-0.3) {};
\node[wvertex] (v6) at (2,0.3) {};
\node[wvertex] (v7) at (2,0.8) {};
\node[wvertex] (v8) at (3,-0.8) {};
\node[wvertex] (v9) at (3,-0.3) {};
\node[wvertex] (v10) at (3,0.3) {};
\node[wvertex] (v11) at (3,0.8) {};
\node[wvertex] (v12) at (4,-0.5) {};
\node[wvertex] (v13) at (4,0.5) {};
\node[wvertex] (v14) at (5,0) {};
\draw (v1)--(v2);
\draw (v1)--(v3);
\draw (v2)--(v4);
\draw (v2)--(v5);
\draw (v2)--(v6);
\draw (v2)--(v7);
\draw (v3)--(v4);
\draw (v3)--(v5);
\draw (v3)--(v6);
\draw (v3)--(v7);
\draw (v4)--(v8);
\draw (v4)--(v9);
\draw (v4)--(v10);
\draw (v5)--(v8);
\draw (v5)--(v9);
\draw (v5)--(v11);
\draw (v6)--(v8);
\draw (v6)--(v10);
\draw (v6)--(v11);
\draw (v7)--(v9);
\draw (v7)--(v10);
\draw (v7)--(v11);
\draw (v12)--(v8);
\draw (v12)--(v9);
\draw (v12)--(v10);
\draw (v12)--(v11);
\draw (v13)--(v8);
\draw (v13)--(v9);
\draw (v13)--(v10);
\draw (v13)--(v11);
\draw (v12)--(v14);
\draw (v13)--(v14);
} \\
     \cline{2-3}
    & $d_x=3$ & none\\
    \cline{2-3}
    & $d_x=4$ & none\\
    \cline{2-3}
    & $d_x=5$ & $Q_5$\\
     \hline
    \end{tabular}
    
    \caption{The list of $C_3$-free Bonnet-Myers Sharp Graphs with $L=2,3,4,5.$}
    \label{table1}
\end{table}

\bibliographystyle{alpha}

\bibliography{BMS}

\end{document}